\documentclass[twocolumn]{autart}
\usepackage{stfloats} 
\usepackage{amsmath,amssymb,amsfonts}
\usepackage{graphicx}
\usepackage{mathptmx} 
\usepackage{epsfig,epstopdf}
\usepackage{bbm}
\usepackage{bm}
\usepackage{mathrsfs}
\usepackage{color}
\usepackage{textcomp}
\usepackage{comment}
\usepackage{varwidth}
\usepackage{tikz}
\usetikzlibrary{arrows,positioning}
\usetikzlibrary{calc}
\tikzset{>=latex}
\usepackage{subfigure}
\usepackage{algorithm,algorithmic}
\usepackage{cite}
\usepackage{subcaption} 

\usepackage[titletoc,title]{appendix}

\newtheorem{assumption}{Assumption}
\newenvironment{proof}[1][Proof]{\noindent\textbf{#1.} }{\ \rule{0.5em}{0.5em}}

\usepackage{hyperref}
\hypersetup{hidelinks}

\begin{document}

        \setlength\abovedisplayskip{15pt}
       \setlength\belowdisplayskip{15pt}
        \setlength\abovedisplayshortskip{15pt}
        \setlength\belowdisplayshortskip{15pt}
        \allowdisplaybreaks
        \setlength{\parindent}{1em}
        \setlength{\parskip}{-0em}  
        \addtolength{\oddsidemargin}{15pt}      
        
        \begin{frontmatter}
                
\title{Multi-layer Predictor Feedback Design for Nonlinear Integro-Differential Equations with State-dependent  Input Delays} 
                
                \thanks[footnoteinfo]{Corresponding author: Mamadou Diagne. 
                }
                
                \author[Tong]{Tong Li}\ead{2011289@tongji.com},  
                \author[Tong]{Peipei Shang}\ead{shang@tongji.edu.cn},
                \author[Mamadou]{Mamadou Diagne*}\ead{mdiagne@ucsd.edu} 
                
                \address[Tong]{School of Mathematical Sciences, Tongji University, Shanghai 200092, China}  
                \address[Mamadou]{Department of Mechanical and Aerospace Engineering, University of California San Diego, La Jolla, CA, USA, 92093} 
                
                

\begin{abstract} 
We develop a novel  multi-layer predictor-feedback to achieve exact compensation of  state-dependent input delay of general nonlinear integro-differential equations. The system of interest is an unconventional \emph{mixed} Partial Differential Equation (PDE)-Ordinary Differential Equation (ODE) system, in which a nonlinear ODE is actuated through an \emph{inhomogeneous} advection PDE. \textcolor{black}{Moreover, the propagation speed of the PDE depends on a moving window integral of the ODE state. The two above features are not addressed yet in standard PDE backstepping-based predictor-feedback designs. } \textcolor{black}{We first address a source term corresponding to a linear recirculation loop. The framework is subsequently generalized to the nonlinear case with the addition of a friction term. \textcolor{black}{In addition in the latter case, we use inflow boundary control with outflow anti-collocated boundary measurements.} In both cases,} we show that the inhomogeneous and \emph{doubly nonlinear} mixed PDE-ODE system results in a correspondingly \emph{``inhomogeneous"} control law. The stabilizing controller comprises a nominal delay-free control law based on the predictor-feedback state, enhanced by a multi-layer prediction mechanism that compensates for the PDE’s nonlinear additive source terms. Our modular designs are based on two novel  nonlinear backstepping transformations: the first enables boundary control of the PDE state, while the second addresses flux control, a nonlinear boundary condition defined as  the product of the ODE and PDE states at the controlled boundary. Unlike the conventional Lyapunov-based approaches  used in the field, our stability and well-posedness analysis rely on  the characteristic method and a fixed-point argument. Both of our designs achieve global asymptotic stability (GAS) in the supremum norm of the PDE and ODE states under the mild assumption that the nonlinearity in the PDE governing equation is uniformly Lipschitz continuous. The transport speed, governed by the integral of the ODE state, models systems such as production or queuing processes in which the state of a finite buffer--namely, the inventory level--determines the production or service rate. \textcolor{black}{Numerical simulations demonstrate the effectiveness of the proposed control design for buffer-regulated production lines and queuing systems, ensuring asymptotic stability under a locally safe softened ``bang–bang" feedback law that preserves the positivity of both the system state and the actuation signal.}
\end{abstract}

      \end{frontmatter}

\section{Introduction}\label{sec:introduction}

\subsection{State of the art of predictor feedback design} The Smith Predictor, introduced in the late 1950s for stable Linear Time-Invariant (LTI) systems, represents a foundational development in compensator design for systems with dead time. In \cite{smith1957closer,smith1959controller}, the concept of prediction is introduced through the use of a delay-free model: the state of the system in the absence of delay is used to estimate its future values and thereby effectively compensate for constant time-delays. An extension of this approach to unstable plants via Finite Spectrum Assignment (FSA) was proposed in \cite{manitius1979finite}. Later, Artstein \cite{artstein1982linear} generalized Smith’s predictor by introducing a predictor-based transformation that reformulates linear or nonlinear systems with constant input delays into equivalent delay-free systems. This transformation facilitates the use of standard state-feedback design methods, even when the open-loop time-delay system is unstable.

Since around 2000, building on the foundational work of Smith and Artstein, PDE backstepping predictor-based control \cite{krstic2008compensating,krstic2010compensation,krstic2009input}, which comes with a stability proof fundamentally rooted in Lyapunov's arguments,
has seen quite interesting progress. The method has driven advances in predictor feedback design for finite-dimensional systems with input delays--ranging from time-varying 
\cite{krstic2010lyapunov,bekiaris2011compensation}, state-dependent \cite{bekiaris2012compensation,Krstic2017X,yu2019bilateral}, and input-history-dependent delays \cite{diagne2017compensation}, to stochastic \cite{kongprediction}, uncertain \cite{bekiaris2010lyapunov}, and distributed  \cite{bekiaris2010lyapunov}. Safety-critical predictor feedback design has been reported in \cite{abel2021safety, abel2019constrained}, while prescribed-time stabilization is addressed in \cite{espitia2021predictor}. Stability-preserving numerical implementations of nonlinear predictor-based schemes can be found in \cite{karafyllis2017predictor}, and their extension to extremum seeking control is developed in \cite{oliveira2016extremum}, with prescribed-time convergence in \cite{yilmaz2024prescribed} and  exponential and unbiased convergence recently proposed in \cite{uESC-PDE}.

Outside the purview of PDE backstepping design, predictor feedback has been widely employed for the stabilization of  strict-feedback nonlinear systems with input delay \cite{cacace2016stabilization}, as well as for stochastic systems subject to input and output delays \cite{cacace2020predictor}. A prediction-based stabilization method for a general class of nonlinear time-varying systems with pointwise input delay is presented in \cite{mazenc2016stabilization}, while sequential predictor feedback strategies are developed in \cite{weston2018sequential, mazenc2017stabilization}, with further extensions to multivariable extremum seeking in \cite{malisoff2020delayed}. Observer-based predictor designs for linear systems with input and state delays are proposed in \cite{zhou2017stabilization}. Interval predictor design is applied in \cite{polyakov2013output} to achieve output feedback stabilization for linear systems with unknown, bounded, time-varying input delays. In addition to predictor feedback approaches, delay-dependent stability conditions have been formulated based on the existence of positive definite solutions to Riccati matrix equations \cite{richard99}. The Lyapunov–Krasovskii functional method has been  employed for the stability analysis of linear delay systems \cite{kolmanovskii1999liapunov} (see \cite{Fridman2003, fridman2003delay}).

 The present work tackles a nonlinear plant with a time-varying delay that depends on both the system's and actuator's states.  Our nonlinear composite PDE-ODE system is equivalent to a \emph{nonlinear integro-differential equation} with an implicitly defined state  dependent input delay.

\subsection{Contributions} 
A key aspect of the backstepping-based predictor feedback design to compensate input delays is the reinterpretation of  time-delays as  advective transport  PDEs, achieved through an infinite-dimensional representation of the actuator state. This approach  transforms a delayed ODE into a nonlinear  PDE coupled with a nonlinear ODE. This reformulation--fundamental to the control design strategies developed in \cite{krstic2010lyapunov,bekiaris2011compensation,bekiaris2012compensation,Krstic2017X,diagne2017compensation,oliveira2016extremum,bekiaris2010lyapunov,karafyllis2017predictor}, where the convective equation has the form of a \emph{conservation law}. In this work, we investigate a less-explored yet important class of advection PDE systems in which the transport speed is given by the integral of the ODE state that is a history of the state. These systems include both linear and nonlinear additive source terms and are coupled  with nonlinear ODEs, where the advection term serves as the input pathway to the ODE dynamics. The stabilization problem concerns a class of nonlinear integro-differential systems with an input delay that is dependent on the integral of the state. The core contributions of this work can be summarized as follows.

 \textbf{(1) Multi-layer predictor feedback design.} We design a multi-layer predictor feedback that achieves asymptotic stabilization and exact compensation of the actuator dynamics for a coupled \emph{doubly nonlinear} coupled PDE–ODE system, where the transport speed depends on the integral of the ODE state 
    {\color{black}indicating that it is the inventory level--rather than the accumulation of workload at production stages--that drives the production speed. This assumption is especially relevant in production or queuing systems where raw materials, semi-finished, or finished products have limited shelf lives \cite{soman2004combined}.
    Moreover, the PDE no longer represents a conservation law but includes additive source terms}, which may be nonlinear functions of its state--for instance, a “recirculation loop” \emph{inhomogeneity} that depend on the uncontrolled boundary. As in prior work, our design necessitates the existence of a nominal controller for the delay-free plant. \textcolor{black}{The  multi-layer predictor feedback design is constructed by   solving the
transport PDE using the method of characteristics (see e.g. \cite{diagne2017compensation,NM2018}). Our controller incorporates a compensator that effectively cancels the influence of inhomogeneous terms in the PDE actuator state, without imposing any restrictions on the severity of the nonlinearity or the \emph{feasibility condition} \cite{bekiaris2011compensation,bekiaris2012compensation}, while guaranteeing the latter through an assumption on the positivity of the transport speed to avoid
     an ill-posed problem (see \cite{Krstic2017X,diagne2017compensation}).} 

  \textbf{(2) Novel nonlinear backstepping transformations.} 
     \textcolor{black}{Following the spirit of \cite{krstic2009delay} and \cite{LMbook}, we use backstepping transformations combined with predictor feedback controls to compensate the delay and make the system stable. But due to the state-dependent delay and source term that appears in the PDE, new transformations rooted in a multi-layer nonlinear predictor whose state accurately forecasts the future trajectories of the nonlinear plant and actuator are introduced. The predictor state is used as the argument of the nominal delay-free control law, thus preserving the structure of predictor-feedback design while compensating for nonlinear, inhomogeneous actuator dynamics. More precisely, the novelty of our work relies on a carefully constructed  stable target system. Furthermore}, our design follows the ``separation principle," combining two distinct compensators: one for the underlying \emph{homogeneous} dynamics, which is described as a conservation law, and another to counteract the inhomogeneous source terms. This framework is first developed for the case of a linear source term, where the second compensator directly neutralizes the ``virtual" disturbance induced by the source terms. For a more general scenario involving a nonlinear source term that is a function of the uncontrolled boundary value, nonlinear boundary conditions (flux control) and an additional linear friction that depends on the PDE state, our  control design approach  is augmented by  an explicit and time-dependent   kernel function that systematically scale the aforementioned compensator components to counteract the effect of the linear friction term.

\textbf{(3) Well-posedness and Stability.} The stability and well-posedness of the closed-loop system present significant technical challenges. The backstepping transformation maps the original system with the predictor-feedback controller into a target system, and is equivalent to its inverse under the supremum norm. Departing from classical Lyapunov methods, we characterize the transformation along characteristic curves to prove global asymptotic stability of the target system, which, by norm equivalence, ensures the stability of the original closed-loop system. 
{\color{black} Though both the well-posedness analysis and stability proof follow the steps in \cite{Li2025, Li2025CPDE}, the technical challenges in the present contribution  are substantially greater. First, an additional weighting function is introduced in order to predict the moving window integral, which complicates the well-posedness analysis. Second,  the doubly nonlinear structure of the system in its PDE and its ODE states under the presence of source terms introduces  significant challenges not present in the simpler cases  \cite{Li2025, Li2025CPDE}. The well-posedness is performed on intricately coupled nonlinear predictor and nonlinear backstepping transformation whereas the stability requires to first prove that the additional predictor layer, designed specifically to compensate for the source term, is itself bounded by the system states.  
when there is no source term in the PDE, \cite{Li2025} has proposed a bilayer predictor-feedback strategy, with one layer predicting the future ODE state and the other addressing the actuator dynamics. However, the presence of source terms requires a tri-layer predictor-feedback design. The additional layer, induced by the source term predicts both the distributed PDE state and the evolution of its uncontrolled boundary.  } 
\textcolor{black}{Our contribution differs from \cite{irscheid2022output,gabriel2023robust,deutscher2018output}, which consider a heterodirectional coupled linear PDE in \emph{cascade} with a nonlinear or linear ODE, and predominantly assume constant coefficients—thereby leading to a constant input delay. Additional control results for PDE–ODE \emph{cascade} systems can be found in \cite{CDC2019,AJ2022,AJ2023,TO2017,TOM2022}. In contrast, our system is not a simple \emph{cascade}: the propagation speed depends on the integral of the nonlinear ODE state, and  source terms are present in the actuator dynamics. This yields a \emph{mixed} PDE–ODE structure with a delay that varies with the ODE state history.
It is also a setting commonly encountered in service queues, supply chains, highly re-entrant manufacturing lines, and other factory production systems \cite{CKW,SW2011}. In our case, flux control leads to a controlled boundary that is nonlinearly influenced by the history of the nonlinear ODE to be stabilized, marking a key distinction from the problem settings in \cite{Li2025,Krstic2017X}. }\\

The organization of this paper is as follows. In Section \ref{section 2}, we present our multi-layer predictor feedback design procedure considering an inhomogeneous PDE with a linear  additive recycling term and state-dependent propagation speed.  In Section \ref{sec3}, we extend the design principle to  a  class of nonlinear actuator dynamics with recycling and friction source terms.  In Section \ref{section m}, we apply our predictor design to a  buffer-driven production line that includes both recycling and friction terms and provide numerical simulations in Section \ref{section 5}. Finally, Section \ref{section con} concludes the paper and outlines future research directions.

\section{Mixed PDE-ODE  under state-dependent transport speed and recycle } \label{section 2}

\subsection{Problem statement and main result}

We consider a coupled PDE-ODE system, where the ODE is given by
\begin{equation}\label{equation X} \dot{X}(t) = f(X(t), u(0,t)), \end{equation}
where $X : [0,\infty) \to \mathbb{R}^n$ denotes the ODE state, and $f : \mathbb{R}^n \times \mathbb{R} \to \mathbb{R}^n$ is a continuously differentiable function satisfying $f(0,0) = 0$. The ODE subsystem is located at the uncontrolled boundary $x = 0$ of the transport domain.

The transport dynamics are governed by the \textcolor{black}{balance law}
\begin{align} \partial_t u(x,t) &= \lambda\left( \int_{t-\tau}^{t} X(s)\,ds \right)\, \partial_x u(x,t) + g(x)\, u(0,t), \label{equation u} \\ 
{u(D,t)}&= {U(t)}, \label{boundary condition} \end{align}
for $(x,t) \in [0,D] \times [0,\infty)$, where $\tau > 0$ is a fixed time window, $g:[0,D] \to \mathbb{R}$ is a continuously differentiable function, and $u : [0,D] \times [0,\infty) \to \mathbb{R}$ represents the actuator state.

The transport speed is described by a continuously differentiable mapping $\lambda : \mathbb{R}^n \to [0,\infty)$. 

The control input is given by the boundary actuation $U(t)$, which will be designed to stabilize the coupled system.

The initial condition along the actuation path is specified by
\begin{equation} u(x,0)=u_0(x),\label{initial condition} \end{equation}
and the initial history of the ODE state is given by
\begin{equation} X(s)=h(s) \end{equation}
for all $-\tau \leq s\leq 0$.

For the dynamics \eqref{equation X}, we make the following basic assumptions:
\begin{assumption}\label{Assumption 1}

	The	system $\dot{X} = f(X, \omega)$ is strongly forward complete with respect to $\omega$.	

\textcolor{black}{
    The system  $\dot{X} = f(X, \omega)$ is strongly forward complete means that, for every initial condition $X(0)$ and every measurable locally essentially bounded function $\omega$, the corresponding solution is defined for all $t\geq 0$ (see \cite{Krstic2010}). }

    \textcolor{black}{Specifically, there exists a smooth positive definite function
	$\Theta $ and class $\mathcal{K}_\infty$ functions $\mathcal{G}_1, \mathcal{G}_2$ and $\mathcal{G}_3$ such that for the
	plant $\dot{X}= f(X,w)$, the following holds (see \cite{Krstic2010}):
	\begin{align}
		\mathcal{G}_1(|X|)\leq &\Theta(X)\leq \mathcal{G}_2(|X|),  \label{equation 7}\\
		\frac{\partial \Theta(X)}{\partial X}f(X,w)\leq &\Theta(X)+\mathcal{G}_3(|w|) \label{equation 8}
	\end{align}
	for all $(X,w)^T\in \mathbb{R}^{n+1}$ .
    }
\end{assumption}

\begin{assumption}\label{Assumption 2}
	The system $\dot{X} = f(X, \kappa(X) + \omega) $ is input-to-state
	stable (ISS) with respect to $\omega$,
	i.e., there exists a class $\mathcal{KL}$ function $\eta$ and a class $\mathcal{K}$ function $\psi$ such that
	\begin{equation}\label{ISSE}
		|X(t)|\leq \eta(|X(0)|,t)+\psi(\|\omega\|),
	\end{equation}
	where $\|\omega\|:=\sup \{|\omega(t)|,\, t\geq 0\}.$

	Moreover, the feedback law
	$\kappa : \mathbb{R}^n \to \mathbb{R} $ is continuously differentiable with $\kappa(0) = 0$. 
\end{assumption}
We adopt the definitions of $\mathcal{K}$, $\mathcal{K}_{\infty}$ and $\mathcal{KL}$  functions, as well as the ISS stability, from \cite{Sontag1995}.

\begin{assumption}\label{Assumption 3}
\textcolor{black}{
To guarantee that the control signal reaches the boundary and ensure the well-posedness of the control law, particularly when the actuation is defined as a flux in Section \ref{sec3}, we assume the existence of two positive constants $\underline \lambda$ and $\bar\lambda$ such that
\begin{equation}\label{eqlam} \underline{\lambda}\leq \lambda(Z) \leq \bar\lambda, \quad \text{for all } Z \in \mathbb{R}^n. \end{equation}
}
\end{assumption}

To stabilize the coupled PDE-ODE system described by \eqref{equation X}--\eqref{boundary condition}, we propose a multi-layer predictor-feedback control law defined by

\begin{align}
U(t)=\underbrace{\kappa(p_1(D,t))}_{\rm \textcolor{blue}{Nominal\, \rm  Control\, \rm  law }}-\underbrace{\int_0^D \frac{g(D-y) p_2(y,t)}{\lambda(p_3(y,t))} dy,}_{\substack{\text{\rm \textcolor{red}{Compensation of the}}\\  \text{\textcolor{red}{ inhomogeneous  term in \eqref{equation u} }}}}\label{boundary control}
\end{align}
where $p_i$, $i=1,2,3$ satisfy the following coupled integral equations
\begin{align}	&p_1(x,t)=X(t)+\int_{0}^{x}\frac{f(p_1(y,t),p_2(y,t))}{\lambda(p_3(y,t))} dy, \label{equation p_1(x,t)}\\
 &p_2(x,t)=u(x,t)+\int_0^x \frac{g(x-y) p_2(y,t)}{\lambda(p_3(y,t))}  dy  \label{equation p_2(x,t)},\\
	&p_3(x,t)=\int_{t-\tau}^{t}\gamma(s,\sigma(x,t))X(s)ds\notag \\&~~~~~~ ~~~~~~~+\int_{0}^{x}\gamma(\sigma(y,t), \sigma(x,t))\frac{p_1(y,t)}{\lambda(p_3(y,t))}dy,  \label{equation p_3(x,t)}\\
	& \sigma(x,t)=t+\int_{0}^{x}\frac{1}{\lambda(p_3(y,t))} dy \label{equation sig(x,t)}
\end{align}
for all $x\in [0,D]$, where the weight function 
\begin{equation}\label{W}
	\gamma(s,t) =\begin{cases}
		1,  \quad & \text{if } t-\tau \leq s \leq t,  \\
		0, \quad & \text{else.  }  
	\end{cases}
\end{equation}

\begin{rem}
The predictor feedback law in \eqref{boundary control} differs significantly from the one proposed in \cite[Chapter 14]{krstic2009delay}, which addresses actuator dynamics governed by a first-order hyperbolic partial integro-differential equation (PIDE) with spatially varying coefficients and an additive integral source term, cascaded with a linear ODE. In that work, the control strategy extends the classical predictor feedback by incorporating a gain kernel under the integral to compensate for the source term's effect. In contrast, the derivation of \eqref{boundary control}–\eqref{W} leads to a more intricate structure due to  the state-dependent advection speed of the PDE as well as the nonlinearity of the ODE plant under consideration  and the inhomogeneity of the infinite-dimensional  actuator state. These challenges necessitate the construction of  a nonlinear and invertible backstepping transformation, as established in \textbf{Lemma \ref{lemma 1}} and \textbf{Lemma \ref{lemma 2}.} 
\end{rem}

{With the predictor-feedback controller, we state our first theorem that guarantees the stability of the resulting closed-loop system, following the PDE backstepping approach, which helps to provide key point of predictor feedback design for the composite PDE-ODE system with source terms.}

\begin{thm}\label{mt}
	Under Assumptions \ref{Assumption 1}--\ref{Assumption 2}, for all initial conditions $h \in C^0([-\tau,0]) $ and Lipschitz continuous  $u_0$ 
	on $[0,D]$ that satisfies the compatibility condition
	\begin{equation}\label{cu}
		u_0 (D) = \kappa(p_1(D,{0}))-\int_0^D \frac{g(D-y) p_2(y,{0})}{\lambda(p_3(y,{0}))} dy,
	\end{equation}
	there exists a unique \textcolor{black}{strong} solution $\left(X,u\right),$ where $X \in  C^1 ([0,\infty))$ and \textcolor{black}{$u$} is locally
	Lipschitz on $[0,D] \times [0,\infty),$ to the
	closed-loop system \eqref{equation X}--\eqref{boundary condition} with the control law (\ref{boundary control})--(\ref{equation sig(x,t)}). Moreover, there exists a class $\mathcal{KL}$
	function $\mathcal{L}_1$ such that the following holds:
	\begin{align}
	   	\sup_{s\in [t-\tau,t]}& \big|X(s)\big|+\sup_{x\in [0,D]}\big|u(x,t)\big| \notag \\
        &\leq \mathcal{L}_1\left(\sup_{s\in [-\tau,0]}\big|h(s)\big|+\sup_{x\in [0,D]}\big|u_0(x)\big|,\,\,t \right) \label{inequality X u} 
	\end{align}
	for all $ t \geq 0$.
\end{thm}

The following lemmas state the backstepping transformation and its inverse.
\begin{lem}\label{lemma 1}
	The control law defined by (\ref{boundary control})--(\ref{equation sig(x,t)}), together
	with the following infinite-dimensional backstepping transformation
	\begin{equation}
		w(x,t)=\underbrace{u(x,t)-\kappa(p_1(x,t))}_{\substack{\text{\rm \textcolor{blue}{Backstepping transformation}}\\  \text{\textcolor{blue}{ \rm for homogeneous  PDEs  }}}}+\underbrace{\int_{0}^{x} \frac{g(x-y)p_2(y,t)}{\lambda(p_3(y,t))}dy}_{\substack{\text{\rm \textcolor{red}{Term induced   }}\\  \text{\textcolor{red}{\rm by a recycle source term}}}} , \label{backstepping transformation}
	\end{equation}
where the three-layer nonlinear state predictor,  $p_i,\, i=1,2,3,$ is defined in  \eqref{equation p_1(x,t)}--\eqref{equation p_3(x,t)},	maps the system \eqref{equation X}--\eqref{boundary condition} into the following target system
	\begin{align}
\dot{X}&=f(X(t),\kappa(X(t))+w(0,t)), \label{equation X w}\\
		\partial_t w(x,t)&=\lambda\left(\int_{t-\tau }^{t}X(s)ds\right)\partial_x w(x,t), \label{equation w}\\
		w(D,t)&=0 \label{boundary condition w}.
        \end{align}
\end{lem}
{\color{black} The compatibility condition for $w$ at $(D,0)$ can be easily checked from  the transformation \eqref{backstepping transformation} and the compatibility condition \eqref{cu}.}
\begin{lem}\label{lemma 2}
	The inverse of the infinite-dimensional backstepping transformation (\ref{backstepping transformation}) is given by
	\begin{equation}\label{inverse backstepping transformation}
u(x,t)=w(x,t)+\kappa(\pi_1(x,t))-\int_0^{x} \frac{g(x-y)\pi_2(y,t)}{\lambda(\pi_3(y,t))}dy,
	\end{equation}
	where $\pi_i$, $i=1,2,3$ are defined as the solutions to the following coupled integral equations
	\begin{align}
&\pi_1(x,t)=X(t)+\int_{0}^{x}\frac{f\left(\pi_1(y,t),\pi_2(y,t)\right)}{\lambda(\pi_3(y,t))}dy, \label{pi1}\\
&\pi_2(x,t)= w(x,t)+\kappa(\pi_1(x,t)),  \label{pi2}\\
		&\pi_3(x,t)= \int_{t-\tau}^{t}\gamma(s, \overline{\sigma}(x,t))X(s)ds \notag \\
          &~~~~~~~~~~~+\int_{0}^{x}  \gamma( \overline{\sigma}(y,t),  \overline{\sigma}(x,t))\frac{\pi_1(y,t)}{\lambda(\pi_3(y,t))}dy,  \label{pi3}\\
		& \overline{\sigma}(x,t)=t+\int_{0}^{x}\frac{1}{\lambda(\pi_3(y,t))} dy,  \label{pi4}
	\end{align}
\end{lem}
{where the weight function $\gamma$ is still defined by \eqref{W}.}
Since {\eqref{equation w}} does not contain a source term, certain components of the \emph{backward predictor} associated with the target system differ from those of the \emph{forward predictor}.
The following lemmas establish the norm equivalence between the original system \eqref{equation X}–\eqref{boundary condition} and the transformed system \eqref{equation X w}–\eqref{boundary condition w}. 

\begin{lem} \label{lemma 3}
	There exists a class $\mathcal{K}_{\infty}$ function $\mathcal{K}_1$ such that
	\begin{equation}\label{norm 1}
		\sup_{x\in [0,D]}\big|w(x,t)\big|\leq \mathcal{K}_1\left( \big|X(t)\big|+\sup_{x\in [0,D]}\big|u(x,t)\big|\right)
	\end{equation}
	for all $t\geq 0$.
\end{lem}
\begin{lem} \label{lemma 4}
	There exists a class $\mathcal{K}_{\infty}$ function $\mathcal{K}_2$ such that
	\begin{equation}\label{norm 2}
		\sup_{x\in [0,D]}\big|u(x,t)\big|\leq \mathcal{K}_2\left( |X(t)|+\sup_{x\in [0,D]}\big|w(x,t)\big|\right)
	\end{equation}
	for all $t\geq 0$.
\end{lem}
{The proofs of Theorem \ref{mt} and the associated Lemmas \ref{lemma 1}--\ref{lemma 4} can be regarded as special cases of {Theorem \ref{thm3} and Lemmas \ref{lemma 9}--\ref{lemma 10} and Lemmas \ref{lemma 11}--\ref{lemma 12} presented in the next section.}

\subsection{Predictor feedback design methodology }

Recalling the method of characteristics, we next present the rationale behind our multi-layer predictor design and explain how it conceptually extends classical delay-compensation schemes to systems with source terms and state-dependent transport speed.

Let us denote by $\zeta(s;x,t)$ the characteristic curve passing through $(x,t)$, i.e.,
\begin{equation}\label{cline}
\zeta(s;x,t)
=x-\int_t^s\lambda\Bigl(\int_{\theta-\tau}^{\theta}X(\eta)\,\mathrm{d}\eta\Bigr)\,d\theta,
\end{equation}
where $X$ is a continuous function.

The input‐delay mapping $\phi$ is then defined implicitly by
\begin{equation}\label{dp}
\zeta\bigl(\phi(x,t);x,t\bigr)=D,
\end{equation}
which means that the characteristic curve passing through $(x,t)$ starts from the boundary $x=D$ at time $\phi(x,t)$.

\textbf{Interpretation of the delay induced by the inhomogeneous PDE.} When the PDE no longer satisfies a conservation law, its output consists not only of the delayed input but also includes contributions from re-entering source terms. Considering the PDE with a recycling-type source term \eqref{equation u}, \eqref{boundary condition} and using the method of characteristics, the following holds
\begin{align}\label{apu0}
     u(0,t) &= U({\phi(0,t)}) + \int_{{\phi(0,t)}}^t g({\zeta(s;0,t)})u(0,s) \, ds, 
\end{align}
where {$\phi(0,t)$ is {defined by \eqref{dp}, i.e.,} the delay induced by the advection PDE satisfying $\zeta\bigl(\phi(0,t);0,t\bigr)=D$}. Equation \eqref{apu0}  indicates that system \eqref{equation X}--\eqref{boundary condition} is no longer equivalent to the following classical time-delay system in \cite{bekiaris2011compensation, bekiaris2012compensation, diagne2017compensation,Krstic2017,Krstic2017u,Krstic2017X}
\begin{equation}\label{ax1}
     \dot{X}(t) = f\left(X(t),\,U(\Phi(t))\right).
\end{equation}
An equivalence can be recovered by modifying  the  input of the ODE to compensate for the effect of the source term. More precisely, we can find  an implicitly defined function $\overline U$ as the ODE input, i.e.,
\begin{equation}\label{apx3}
    \dot{X}(t) = f\left(X(t),\,\overline{U}(\Phi(t)) \, \right).
\end{equation}
In \eqref{ax1} and \eqref{apx3}, $\Phi(t)$ refers to the delay which is $\phi(0,t)$ using our notation. 
\textcolor{black}{Moreover, the function $\overline{U}$ represents the \textit{effective input} acting on the ODE, which encapsulates both the delayed boundary control $U$ and the accumulated contribution of the re-entering source terms along the characteristic. Crucially, the delay argument $\Phi(t) = \phi(0,t)$ is determined exclusively by the transport speed, which in turn depends on the history of the ODE state but is independent of the actuator state. This clarifies that the system \eqref{equation X}--\eqref{boundary condition} retains the fundamental structure of a system with state-dependent (rather than input-dependent) delay, despite the presence of distributed source terms.}

\textbf{Interpretation of the compensator design.} To compensate the state‐dependent delay induced by the hyperbolic transport PDE, we construct the predictor along the characteristic curves. 
To that end, for any fixed $(x,t)\in[0,D]\times[0,\infty)$, we define the prediction instant $\sigma(x,t)$ by $$\zeta\bigl(\sigma(x,t);x,t\bigr)=0,$$ i.e., the time that the characteristic curve passing through $(x,t)$ arrives at the boundary $x=0$.
\begin{figure}[t]
  \centering
  \includegraphics[width=0.8\linewidth]{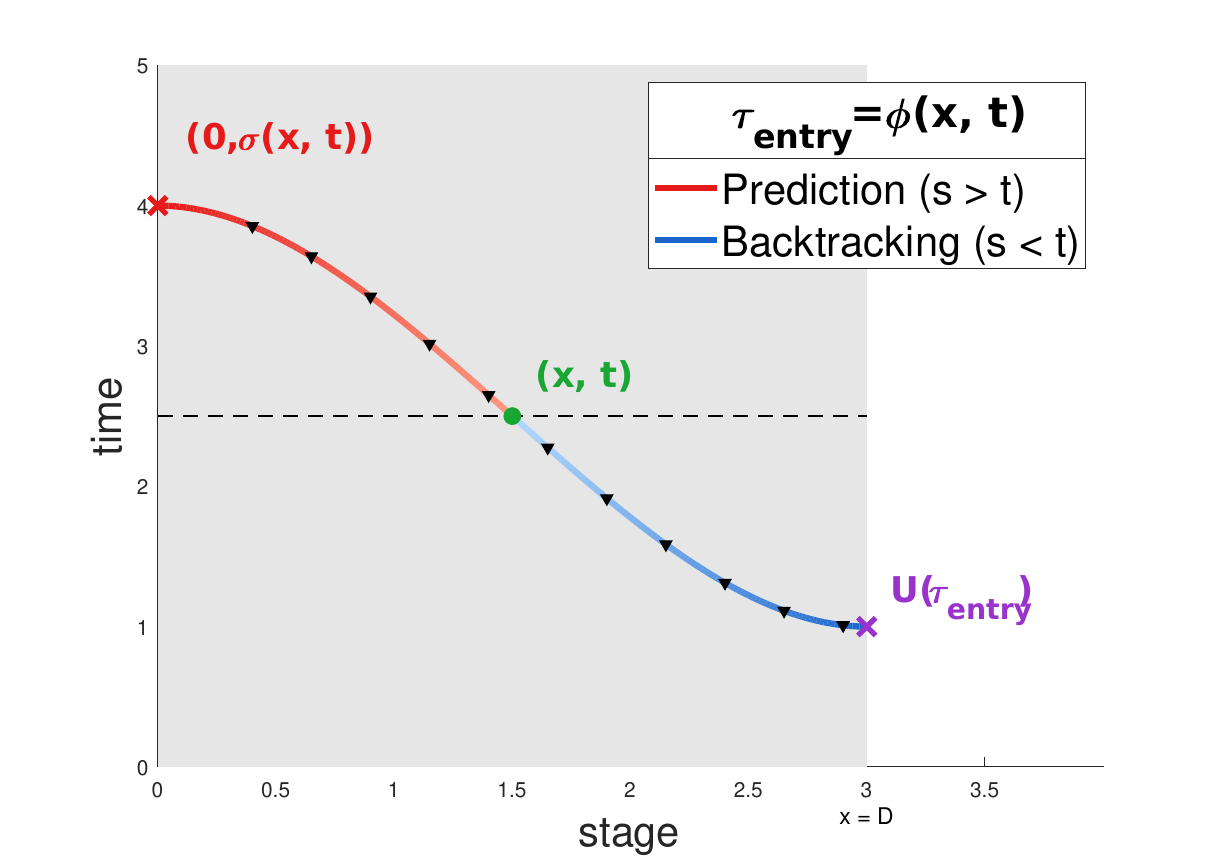}
  \caption{Characteristic curve diagram; the color gradient indicates the variation in control magnitude during transport, and the arrowheads denote the direction of propagation along each characteristic.}
  \label{fig:characteristic-line}
\end{figure}
{
A schematic of these characteristic curves is shown in Fig.~\ref{fig:characteristic-line}.}
Along each characteristic curve we introduce a triad of predictor variables:
\begin{align}
    p_1(x,t)&=X(\sigma(x,t)),\label{rp1}\\
    p_2(x,t)&=u(0, \sigma(x,t)),\label{rp2}\\
    p_3(x,t)&= \int_{\sigma(x,t)-\tau}^{\sigma(x,t)} X(\eta)\,d\eta. \label{rp3}
\end{align}
These are obtained by inverting the nonlinear coupled  PDE–ODE system  \eqref{equation X}–\eqref{boundary condition} along the characteristics, and each {plays a distinct role} in the feedback law:
\begin{itemize}
    \item 
the state  $p_1$  predicts the ODE state to compensate for the input delay; \item the  state $p_2$ reconstructs the effective control signal arriving at the boundary to cancel the influence of the source term, $g(x)u(0,t)$;
 \item the state $p_3$ captures the actuator dynamics and enables the transformation of time-domain predictions into spatial integrals computed along the characteristic curves of the transport PDE.
\end{itemize}
\begin{rem}
    The resulting boundary control law \eqref{boundary control} combines the nominal delay-free feedback $\kappa$ evaluated at $p_1(D,t)$, with an integral correction term that accounts for the re-entrant source effect, based on $p_2$. 
In the special case where $g \equiv 0$ and {$\lambda=v$ is a constant}, the {prediction instant} simplifies to \textcolor{black}{$\sigma(x,t) = t + x/v$}, and the predictor-feedback structure reduces to a classical finite-dimensional delay compensator such as the Smith predictor or the Artstein transformation. In a broad sense, our designs remain valid for PDEs like \eqref{equation u}, \eqref{boundary condition} with time-dependent or spatially-varying transport speed. In such cases, the variation in $\lambda$ alters the mapping $\sigma(x,t)$ and the resulting effective delay {$\phi(0,t)$}, but the  structure of the control law remains unchanged. 
\end{rem}

	\section{Generalization to nonlinear source terms and boundary flux control}\label{sec3}

In this section, we pursue a broader generalization by extending our design to systems with inflow boundary control and outflow anti-collocated boundary measurements. This configuration is not only practically relevant but also introduces significant challenges, as both the input and output are nonlinear functions defined by the product of the state integral and the PDE’s boundary value. Instability in the ODE dynamics inevitably manifests through the nonlinearly defined boundary conditions. Furthermore, we incorporate friction-type terms and nonlinear, spatially dependent recycling functions as additive source terms. Consider the following system
\begin{align}
\dot{X}(t) &= f(X(t), q(t)), \label{M1}\\
q(t)&=\lambda\left( \int_{t-\tau}^{t} X(s)\,ds \right) u(0,t),  \label{M2}\\
	\partial_t u(x,t) &= \lambda\left( \int_{t-\tau}^{t} X(s)\,ds \right)\, \partial_x u(x,t) \notag \\&~~~~~~+ g(x, u(0,t)) + c(x)u(x,t), \label{M3}\\
    U(t)&=\lambda\left( \int_{t-\tau}^{t} X(s)\,ds \right) u(D,t), \label{M4}
\end{align}
where  $c : [0,D] \to \mathbb{R}$ is a bounded, smooth coefficient function. We state the following standard regularity conditions on the nonlinear source term $g(x,v)$.

\begin{assumption}\label{assump3}
The function $g : [0,D] \times \mathbb{R} \to \mathbb{R}$ is continuously differentiable. It is also uniformly Lipschitz continuous with respect to its second argument, i.e., there exists a constant $L_g > 0$ such that for all $x \in [0,D]$ and any $v_1, v_2 \in \mathbb{R}$:
\begin{equation}
|g(x, v_1) - g(x, v_2)| \leq L_g |v_1 - v_2|.
\end{equation}
Furthermore, reflecting the physical nature of many recycling systems, we assume that 
\begin{equation}
    g(x,0)=0, \quad \text{for all } x \in [0,D].
\end{equation}
\end{assumption}
{
\begin{rem}[On the Structure of the Source Terms]
The structure of the source terms in \eqref{M3} is key to the control design. Their independence from the ODE state 
\textcolor{black}{$X$} prevents a destabilizing feedback loop that could violate forward completeness (Assumption~\ref{Assumption 1}). Additionally, the linear friction term in 
\textcolor{black}{$
u$} ensures that the backstepping kernel \eqref{MF7} admits an explicit exponential form; a nonlinear friction term  requires solving an intractable nonlinear PDE.
\end{rem}
}


The predictor-feedback control law for the plant \eqref{M1}--\eqref{M4} is designed as follows:
\begin{align}
 U(t)&=\frac{\lambda(p_3(0,t))\, \kappa(p_1(D,t))}{\lambda(p_3(D,t)) \,       K(D,D,t) }   \notag \\
&~~~~~ -\int_0^D \frac{\lambda(p_3(0,t))}{K(D,y,t)}\,  \frac{g(D-y,  \, p_2(y,t) )} {\lambda(p_3(y,t))}  dy,  \label{MF1}  \\
	p_1(x,t) &= X(t) + \int_0^x \frac{f(p_1(y,t),\, \lambda(p_3(y,t)) p_2(y,t))}{\lambda(p_3(y,t))} \, dy, \label{MF3}   \\  
	p_2(x,t) &= K(x,x,t) \cdot u(x,t) 
	 \notag \\  
	&~~~~+ \int_0^x \frac{K(x,x,t)}{K(x,y,t)} \cdot \frac{g(x-y,\,p_2(y,t))}{\lambda(p_3(y,t))} \, dy, \label{MF4} \\  
	p_3(x,t) &= \int_{t-\tau}^t \gamma(s, \sigma(x,t))\, X(s)\, ds 
	 \notag \\
	&~~~~ + \int_0^x \gamma(\sigma(y,t), \sigma(x,t)) \cdot \frac{p_1(y,t)}{\lambda(p_3(y,t))} \, dy, \label{MF5} \\  
	\sigma(x,t) &= t + \int_0^x \frac{1}{\lambda(p_3(y,t))} \, dy,  \label{MF6}
\end{align}
where the weight function $\gamma$ is still defined by \eqref{W} and the kernel  is  defined by
\begin{align}
	K(x_1,x_2,t) := \exp\left( \int_0^{x_2} \frac{c(x_1 - z)}{\lambda(p_3(z,t))} \, dz \right) \label{MF7}
\end{align}
 on the domain $\mathcal{T}:=
\left\{(x_1,x_2,t),   0\leq x_2 \leq x_1 \leq D, \,t\geq 0 \right\}$.

\begin{figure}[t]
  \centering
  \includegraphics[width=0.8\linewidth]{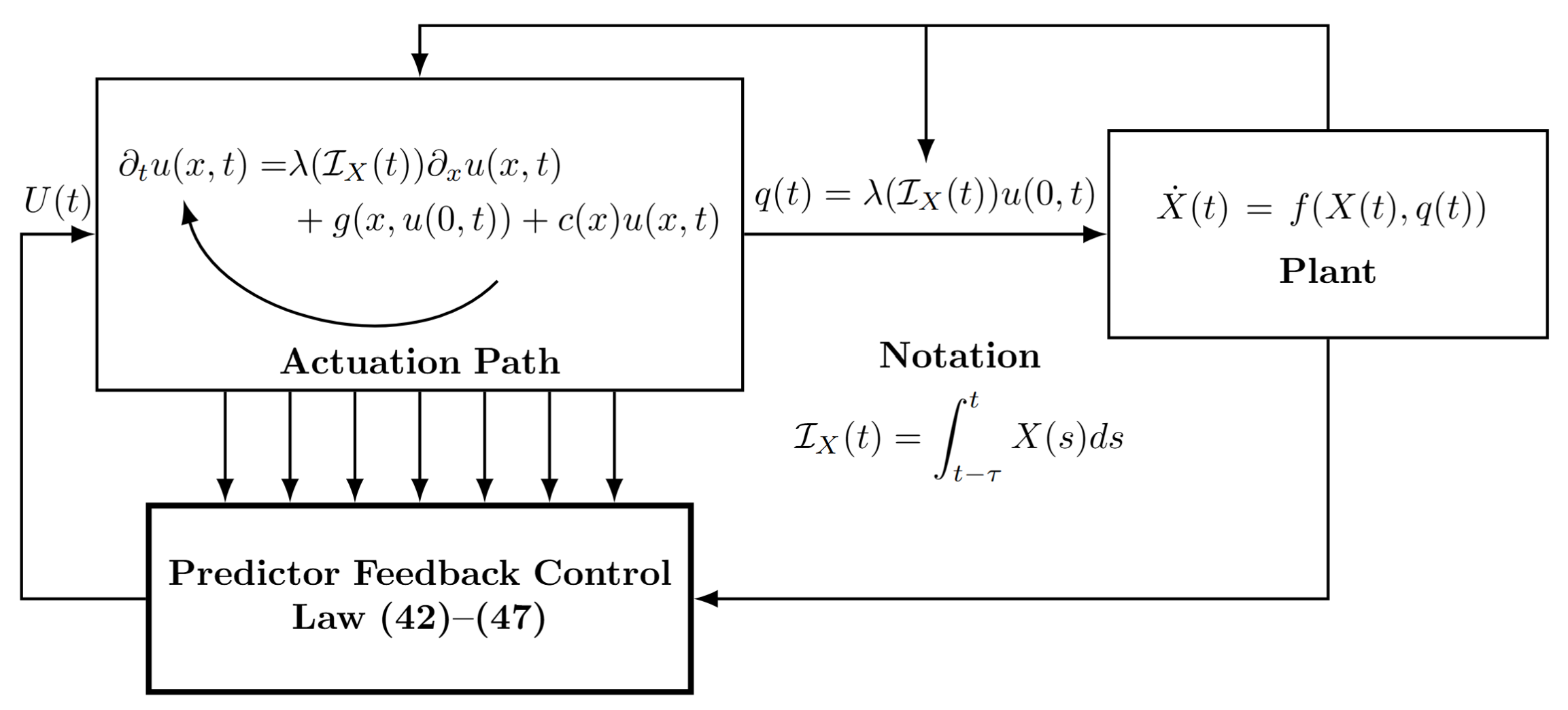}
  \caption{Schematic of the closed-loop system }
  \label{fig:characteristic-line}
\end{figure}

{Under the layer predictor feedback control law \eqref{MF1}--\eqref{MF7},} we state the following stability theorem.
\begin{thm}\label{thm3}
	Under Assumptions \ref{Assumption 1}--\ref{assump3}, for all initial conditions $h \in C^0([-\tau,0]) $ and Lipschitz continuous  $u_0$ 
	on $[0,D]$ that satisfies the compatibility condition
	\begin{align}\label{nu0}
		u_0(D)=&\frac{1}{K({D},{D},0)} \frac{\kappa(p_1(D,0))}{\lambda(p_3(D,0))} \notag \\\
    & -\int_0^{{D}} \frac{1}{K({D},y,0)} \frac{g({D}-y,\, p_2(y,0))}{\lambda(p_3(y,0))}  dy,
	\end{align}
	{there exists a unique \textcolor{black}{strong} solution $\left(X,u\right),$} where $X \in  C^1 ([0,\infty))$ and \textcolor{black}{$u$} is locally
	Lipschitz on $[0,D] \times [0,\infty),$ to the
	closed-loop system \eqref{M1}--\eqref{M4} with the control law \eqref{MF1}--\eqref{MF7}. Moreover, there exists a class $\mathcal{KL}$
	function {$\mathcal{L}_2$} such that the following holds:
	\begin{align} \label{non-sta}
	   	\sup_{s\in [t-\tau,t]}& \big|X(s)\big|+\sup_{x\in [0,D]}\big|u(x,t)\big| \notag \\
        &\leq {\mathcal{L}_2}\left(\sup_{s\in [-\tau,0]}\big|h(s)\big|+\sup_{x\in [0,D]}\big|u_0(x)\big|, \,t \right)  
	\end{align}
	for all $ t \geq 0$.
\end{thm}

The proof of Theorem~\ref{thm3} is mainly related to the following lemmas. Indeed, the closed-loop system \eqref{M1}--\eqref{M4} under feedback control law \eqref{MF1}--\eqref{MF7}}  is mapped into a desired  target system  via an infinite-dimensional backstepping transformation as follows.

\begin{lem}[Backstepping Transformation]\label{lemma 9}
	The control law defined by  \eqref{MF1}--\eqref{MF7}, together
	with the following infinite-dimensional backstepping transformation
\begin{align}\label{nonlinear backstepping}&w(x,t)=\lambda(p_3(x,t))K(x,x,t) u(x,t)-\kappa(p_1(x,t)) \notag \\
    & +\lambda(p_3(x,t))\int_0^x \frac{K(x,x,t)}{K(x,y,t)}\cdot \frac{g(x-y,\, p_2(y,t))}{\lambda(p_3(y,t))}  dy,
	\end{align}
	maps the system \eqref{M1}--\eqref{M4} into the following target system
	\begin{align}
\dot{X}&=f(X(t),\kappa(X(t))+w(0,t)), \label{non tar 1}\\
		\partial_t w(x,t)&=\lambda\left(\int_{t-\tau }^{t}X(s)ds\right)\partial_x w(x,t), \label{non tar 2}\\
		w(D,t)&=0 \label{non tar 3}.
        \end{align}
\end{lem}
\begin{lem}[Inverse Backstepping Transformation]\label{lemma 10}
	The inverse of the infinite-dimensional backstepping transformation \eqref{nonlinear backstepping}, which maps the system \eqref{non tar 1}--\eqref{non tar 3} to the closed-loop system \eqref{M1}--{\eqref{M4}},  is given by
    \begin{align}\label{non inverse backstepping}
&u(x,t)=\frac{L(x,x,t)}{\lambda(\pi_3(x,t))} \left(w(x,t)+\kappa(\pi_1(x,t))\right)\notag \\
    &~~~~~~~  -\int_0^x L(x,y,t) \frac{g(x-y, \, \pi_2(y,t))}{\lambda(\pi_3(y,t))}  dy ,
	\end{align}
	where $\pi_i, \, i=1,2,3$ are defined as the solutions to the following coupled integral equations
    \begin{align}
	&\pi_1(x,t)\!\! =\! X(t)\! +\! \int_0^x \frac{f(\pi_1(y,t),\, \lambda(\pi_3(y,t)) \pi_2(y,t))}{\lambda(\pi_3(y,t))} \, dy,   \label{npi1}\\  
	&\pi_2(x,t) = \frac{w(x,t)+\kappa(\pi_1(x,t))}{\lambda(\pi_3(x,t))} , \label{npi2} \\  
	&\pi_3(x,t) = \int_{t-\tau}^t \gamma(s, \overline{\sigma}(x,t))\, X(s)\, ds 
	 \notag \\
	&\quad + \int_0^x \gamma(\overline{\sigma}(y,t), \overline{\sigma}(x,t)) \cdot \frac{\pi_1(y,t)}{\lambda(\pi_3(y,t))} \, dy, \label{npi3} \\  
	&\overline{\sigma}(x,t) = t + \int_0^x \frac{1}{\lambda(\pi_3(y,t))} \, dy,  \label{npi4}
\end{align}
with the kernel function defined on the domain $\mathcal{T}$ as
\begin{align}\label{kernel L}
	L(x_1,x_2,t) := \exp\left( -\int_0^{x_2} \frac{c(x_1 - z)}{\lambda(\pi_3(z,t))} \, dz \right). 
\end{align}
\end{lem}
{
\begin{rem}
The nonlinear backstepping transformations \eqref{nonlinear backstepping}  generalize \eqref{backstepping transformation}  in two key ways. They incorporate the predicted propagation speed $\lambda(p_3(x,t))$  to account for the flux-based system \eqref{M1}--\eqref{M4}, and use a state-dependent exponential kernel \textcolor{black}{$K$} to cancel the friction term \textcolor{black}{$c\cdot u$} that is absent in \eqref{equation u}. Therefore, the resulting Volterra-type transformation \eqref{nonlinear backstepping}  preserve a homogeneous transport structure in the target system \eqref{non tar 1}--\eqref{non tar 3}.
\end{rem}
}

Next, we establish the norm equivalence between the original system and that of the target system.
\begin{lem}[Norm Equivalence] \label{lemma 11}
	There exists a class $\mathcal{K}_{\infty}$ function {$\mathcal{K}_3$} such that
	\begin{equation}\label{norm 5}
		\sup_{x\in [0,D]}\big|w(x,t)\big|\leq {\mathcal{K}_3}\left( \big|X(t)\big|+\sup_{x\in [0,D]}\big|u(x,t)\big|\right)
	\end{equation}
	for all $t\geq 0$.
\end{lem}

\begin{lem}[Inverse Norm Equivalence] \label{lemma 12}
	There exists a class $\mathcal{K}_{\infty}$ function {$\mathcal{K}_4$} such that
	\begin{equation}\label{norm 6}
		\sup_{x\in [0,D]}\big|u(x,t)\big|\leq {\mathcal{K}_4}\left( |X(t)|+\sup_{x\in [0,D]}\big|w(x,t)\big|\right)
	\end{equation}
	for all $t\geq 0$.
\end{lem}
{
The proofs of Lemmas \ref{lemma 9}--\ref{lemma 10} and Lemmas \ref{lemma 11}--\ref{lemma 12} are provided in Appendix~\ref{Appendix A}. 
We are in a position to prove the main result, Theorem  \ref{thm3}.
}

\begin{proof}
The well-posedness of the closed-loop system is established in Appendix \ref{Appendix B}. Specifically, the analysis in Appendix \ref{Appendix B} guarantees that the target system \eqref{non tar 1}--\eqref{non tar 3} admits a unique, locally Lipschitz solution $w$, given explicitly by:
\begin{equation} \label{wsolution}
w(x,t) = 
\begin{cases} 
w_0(\zeta(0;\,x,t)), & \text{if } 0 \le \zeta(0;\,x,t) \le D, \\
0, & \text{if } \zeta(0;\,x,t) > D.
\end{cases}
\end{equation}
\textcolor{black}{The compatibility condition of the solution \eqref{wsolution} at $(D,0)$ can be easily checked from the transformation \eqref{nonlinear backstepping} and the compatibility condition \eqref{nu0}.}
The subsequent development of this section is concerned with the proof of the stability estimate \eqref{non-sta}. From the explicit solution \eqref{wsolution}, we directly obtain 
\begin{equation}\label{sup w1}
	\sup_{x\in [0,D]} \big|w(x,t)\big| \leq \sup_{x\in [0,D]} \big|w_0(x)\big|, \quad \text{for all} \quad t\geq 0,
\end{equation}
and the fact that it vanishes in finite time:
\begin{equation}\label{sup w2}
	|w(x,t)|=0 \quad \text{for} \quad t > \sigma(D,0).
\end{equation}

		Based on Assumption \ref{Assumption 2}, there exist a class $\mathcal{KL}$ function $\widetilde{\mathcal{L}}_1$
		and a class $\mathcal{K}_{\infty}$ function $\widetilde{\mathcal{K}}_1$ such that the solution of {(\ref{non tar 1})} satisfies
		\begin{equation}\label{ISSX}
			|X(t)|\leq \widetilde{\mathcal{L}}_1( |X(s)|,t-s)+\widetilde{\mathcal{K}}_1\bigg(\sup_{\tau \in [s ,t]}|w(0,\tau)|\bigg) 
		\end{equation}
		for all $t\geq s \geq 0$. We divide the time domain into different intervals and give the estimate respectively.		
		
		For $0\leq t \leq \sigma(D,0)$, since the initial condition $X(s)=h(s), s\in [-\tau,0]$,  from \eqref{sup w1}, \eqref{ISSX},  we have 
		\begin{align}
			&\sup_{s\in [t-\tau,t]}|X(s)| \leq \sup_{s\in [-\tau,0]}|h(s)|+ \sup_{s\in [0,t]}|X(s)| \notag \\
			&\leq \sup_{s\in [-\tau,0]}|h(s)| + \widetilde{\mathcal{L}}_1\left( |h(0)|,0\right)+\widetilde{\mathcal{K}}_1\left(\sup_{\tau \in [0,t]}|w(0,\tau)|\right) \notag \\
			& \leq \widetilde{\mathcal{K}}_2\left(  \sup_{s\in [-\tau,0]}|h(s)|+\sup_{x \in [0,D]}|w_0(x)|\right),  \label{supX1}
		\end{align}		
		where $\widetilde{\mathcal{K}}_2(s)=s+\widetilde{\mathcal{L}}_1(s,0)+\widetilde{\mathcal{K}}_1(s)$.
		
		For $t  \geq \sigma(D,0)+\tau$, from \eqref{sup w2}, \eqref{ISSX}, we have 
	\begin{equation}\label{esbsig}
			\sup_{s\in [t-\tau,t]}|X(s)| \leq  \widetilde{\mathcal{L}}_1\left(\, |X(\sigma(D,0))|,\,\,t-\tau -\sigma(D,0)\,\right), 
		\end{equation}		
		 using \eqref{sup w1}, \eqref{ISSX} again gives 
		\begin{align}\label{esig}
			|X(\sigma(D,0))| \leq \widetilde{\mathcal{L}}_1(|h(0)|, \sigma(D,0))+\widetilde{\mathcal{K}}_1\left(\!\!\sup_{x \in [0,D]}|w_0(x)|\!\!\right). 
		\end{align}
		Substituting  \eqref{esig} into \eqref{esbsig} and recalling the property of $\mathcal{KL}$ function, there exists a class $\mathcal{KL}$ function $\widetilde{\mathcal{L}}_2$ such that
		\begin{align}
			\sup_{s\in [t-\tau,t]}|X(s)| &\leq \widetilde{\mathcal{L}}_1\Bigg(\!\!\! \widetilde{\mathcal{L}}_1\left(|h(0)|, \sigma(D,0)\right)\notag \\ 
            &~~~~~+\widetilde{\mathcal{K}}_1\left(\!\!\sup_{x \in [0,D]}|w_0(x)|\!\!\right),\,t-\tau -\sigma(D,0)\!\!\!\Bigg)
            .\label{supX2}
		\end{align} 
		
		For $ \sigma(D,0)+\tau > t > \sigma(D,0)$, from \eqref{supX1}--\eqref{esig}, we have 
		\begin{align}
			&\sup_{s\in [t-\tau,t]}|X(s)| \leq \sup_{s\in [t-\tau,\sigma(D,0)]} |X(s)|+ \sup_{s\in [\sigma(D,0),t]}| X(s)| \notag \\ 
			& \leq \widetilde{\mathcal{K}}_2\bigg(  \sup_{s\in [-\tau,0]}|h(s)|+\sup_{x \in [0,D]}|w_0(x)|\bigg) \notag \\
            &\quad + \widetilde{\mathcal{L}}_1 \left(\widetilde{\mathcal{L}}_1\left(|h(0)|, \sigma(D,0)\right)+\widetilde{\mathcal{K}}_1\left(\!\!\sup_{x \in [0,D]}|w_0(x)|\!\!\right),0\right). \label{supX3}
		\end{align}		
		Combining the estimates  \eqref{supX1}, \eqref{supX2} and \eqref{supX3},  from the property of $\mathcal{KL}$, $\mathcal{K}_{\infty}$ functions,  we can conclude the existence of a class $\mathcal{KL}$ function $\widetilde{\mathcal{L}}_3$ such that for all $t\geq 0$
		\begin{align}\label{normLy}
			\sup_{s\in [t-\tau,t]}|X(s)| 
			\leq  \widetilde{\mathcal{L}}_3 \left(\sup_{s\in [-\tau,0]}|h(s)|+\sup_{x \in [0,D]}|w_0(x)|,t\right). 
		\end{align}
        Moreover, from \eqref{sup w1}, \eqref{sup w2}, we can find a $\mathcal{KL}$ function $\widetilde{\mathcal{L}}_4$ such that 
        \begin{equation}\label{KLW}
           \sup_{x \in [0,D]}|w(x,t)| \leq  \widetilde{\mathcal{L}}_4\left(\sup_{x \in [0,D]}|w_0(x)|,t \, \right).
        \end{equation}
		Therefore, from \eqref{norm 5}, \eqref{norm 6}, \eqref{normLy} and \eqref{KLW}, we deduce that 
		\begin{align}
\mathcal{L}_2(s,t)=&\widetilde{\mathcal{L}}_3\left( s+\mathcal{K}_3(s),t \right) \notag \\
&~~~~~~~+\mathcal{K}_{4}\left(\widetilde{\mathcal{L}}_3\left( s+\mathcal{K}_3(s),t \right)+\widetilde{\mathcal{L}}_4(\mathcal{K}_3(s),t)\right),
		\end{align}
		which is $\mathcal{KL}$ function from the definition directly. Therefore the proof of estimation {\eqref{non-sta}} is complete.
	\end{proof}
    

    

\section{Modeling of a buffer-driven  conveyor belt with source terms and backlog}  \label{section m}
{We consider a buffer-driven production line or service queuing system extending the work of  \cite{Krstic2017u} to systems with transport speed depending on the real-time  inventory level.  In the case of a factory line, raw material enters the conveyor at its downstream end $x=D$, travels upstream to $x=0$, and then exits the production line into the buffer.  }

\subsection{Governing equations of the buffer ODE  and conveyor PDE }

\textbf{Buffer ODE influx form.} We denote the load of goods stored in the queue by $Q(t)\geq 0$ to represent the  buffer occupancy. In particular, from the conservation of mass, the buffer evolves according to
\begin{equation}\label{eq buffer1}
\dot Q(t)
=\nu_{\mathrm{in}}(t)
-\nu_{\mathrm{out}}(t),
\end{equation}
where the influx of raw material  
\begin{equation}\label{eq buffer2}
\nu_{\mathrm{in}}(t)
=\alpha  \varphi(t).
\end{equation}
and the outflow  retains the standard service‐rate form
\begin{equation}\label{eq buffer3}
\nu_{\mathrm{out}}(t)
=\min\{\,Q(t),\,\mu\}.
\end{equation}

Here:
\begin{itemize}
  \item $\varphi(t)$ denotes the material flux exiting the production line at $x=0$, which is available directly from the predictor;
  \item $\mu>0$ is the maximum service rate of the buffer;
  \item $\alpha>0$ is the connectivity coefficient between buffer and the production line; The connectivity coefficient denotes the fraction of materials flux flowing from the final stage to the buffer when the  rate of losses is assumed to be known and equal to $1-\alpha$.
\end{itemize}

\textbf{Queuing systems model.}
We consider a production‐line model where the  density  of parts at stage $0\leq x \leq D$ and time $ t\geq 0$ denoted by $\rho(x,t)$,  satisfies the following PDE
\begin{equation}\label{Mp}
\partial_t \rho(x,t)
=\omega\bigl(R(t)\bigr)\,\partial_x \rho(x,t)
+S(x,t)
-C(x)\,\rho(x,t),
\end{equation}
where  $\omega\bigl(R(t)\bigr)$ is defined by 
\begin{equation}
 \omega(R)
=\frac{1}{P\,(1+R)} ,
\qquad R(t)=\int_{t-\tau}^{t}Q(s)\,\mathrm{d}s
,  
\end{equation}
with $P$ defined as the processing time and $\tau,$   a chosen constant that indicates  the effect of the history of the buffer load on the speed of production.
The  output flux of the production line \eqref{Mp} at $x=0$ is defined as follows 
\begin{equation}\label{Mp1}
    \varphi(t)= \omega(R(t)) \rho(0,t),
\end{equation}
and influx at the  controlled boundary, $x=D,$ is given by
\begin{equation}\label{Mpd}
    U(t)= \omega(R(t))  \rho(D,t).
\end{equation} 

{
The term $S$ in \eqref{Mp} models  re-injection of reworked parts back into the production line, namely, the recycling effect
\begin{equation}\label{Mp3}
    S(x,t)=A\,(D - x)\,\rho(0,t),
\end{equation}
where $\rho(0,t)$ is the part density value at the exit of the line and $A>0$ a factor of the total rework rate at each production stage, \textcolor{black}{which should be properly chosen according to real applications (see Proposition \ref{r16}).}   

{
The friction term 
\begin{equation}\label{Mp4}
C(x)\,\rho(x,t),
\end{equation}
with $C(x)\geq 0$, captures losses during part processing (e.g.\ scrap or wear) along the conveyor.}

{
In summary, the model \eqref{Mp}–\eqref{Mp4} extends the classical transport equation by (i) a state‐dependent convection speed $\omega(R(t))$ that reflects buffer congestion, (ii) a recycling source $S$ linked to lost material recovery, and (iii) a friction term $\textcolor{black}{C\cdot \rho}$ accounting for in‐process losses.}

By using the $\varphi(t)$, we avoid the density‐based transformation employed in prior work (e.g.\ \cite{Borsche2010,Herty2007,Krstic2017u}, and instead feed the predicted flux directly into \eqref{eq buffer1}.  This simplifies implementation and improves numerical efficiency.

\subsection{Safe softened ``bang-bang" control of the delay-free plant }
Before introducing the delay compensation, we briefly recall the nominal feedback control law $u(Q^*,Q(t))$ for the delay‐free buffer dynamics of \eqref{eq buffer1}--\eqref{eq buffer3}, as developed in \cite{Krstic2017u}
\begin{equation}\label{Mp5}
\dot Q(t)
=\alpha\,u(Q^*,Q(t))
-\min\{Q(t),\mu\}.
\end{equation}
  Our objective is to regulate $Q(t)$ to a prescribed set‐point $Q^*\in[0,\mu]$.

We adopt a piecewise‐exponential “bang–bang” law of the form  
\begin{align}\label{bangbang}
u(Q^*,Q(t))
&=B_\ell\bigl(Q(t),Q^*\bigr)\,\mathcal{H}\bigl(Q^*-Q(t)\bigr)
\notag \\
&~~~~~~~~~+ B_r\bigl(Q(t),Q^*\bigr)\,\mathcal{H}\bigl(Q(t)-Q^*\bigr),
\end{align}
where $\mathcal{H}$ denotes the Heaviside function.  The branches $B_\ell$ and $B_r$ are defined so as to match the value and slope at $Q=Q^*$:
\begin{align}\label{bangBl}
B_\ell(Q,Q^*)
&=\frac{Q^*}{\alpha} + \bigl(B_{\max}-\frac{Q^*}{\alpha} \bigr)\,
\frac{1 - e^{\Lambda_\ell(Q^*)\,(Q-Q^*)}}
     {1 - e^{-\Lambda_\ell(Q^*)\,Q^*}},
\end{align}
\begin{align}\label{bangBr}
B_r(Q,Q^*)
=\frac{Q^*}{\alpha} -\frac{Q^*}{\alpha}
\frac{1 - e^{-\Lambda_r(Q^*)\,(Q-Q^*)}}
     {1 - e^{-\Lambda_r(Q^*)\,(\,Q_{\max}-Q^*\,)}}.
\end{align}
Here $u$ is the input {flux} and $B_{\max}$ and $Q_{\max}$ are the maximum value of the input {flux} and capacity of the queue, respectively. In order to guarantee a well-running process, we require that the maximum input remains below both the system service capacity and the queue upper bound $Q_{\max}$, which means that these maximum values are chosen to satisfy
\begin{equation}
B_{\max} \le \frac{ \min\{Q_{\max}, \mu\}}{\alpha}.
\end{equation}
The flux-based representation of the ODE dynamics reduced to \eqref{Mp5} does not require an explicit constraint on 
$Q_{\max},$ thereby lifting the limitation 
$Q_{\max} < \alpha/P$ imposed in \cite{Krstic2017u}.

The positive gains $\Lambda_\ell(Q^*)$, $\Lambda_r(Q^*)$
are uniquely determined to ensure the
continuous differentiability of the extended control law \eqref{bangbang}, namely, they are chosen to satisfy  
\begin{equation}
S(Q^*):=- \frac{\partial B_\ell (Q,Q^*) }{\partial Q}\bigl|_{Q=Q^*}
=-\frac{\partial B_r (Q,Q^*)}{\partial Q}\bigl|_{Q=Q^*},
\end{equation}
where $S(Q^*)$ is  the slope design parameter of the controller satisfied
\begin{align}
S(Q^*) &= \frac{\Lambda_\ell(Q^*) \left( B_{\max} - \frac{Q^*}{\alpha} \right)}{1 - e^{-\Lambda_\ell(Q^*) Q^*}}, \label{sq1}\\
S(Q^*) &= \frac{\Lambda_r(Q^*) \, }{1 - e^{-\Lambda_r(Q^*) (Q_{\max} - Q^*)}} \frac{Q^*}{\alpha}. \label{sq2}
\end{align}
The right hand sides of \eqref{sq1} and \eqref{sq2}  are deduced by differentiating \eqref{bangBl} and \eqref{bangBr} with respect to the state $Q$ for $Q = Q^*$. The gains $\Lambda_\ell(Q^*)$ and $\Lambda_r(Q^*)$ are computed as the unique strictly positive solutions of the fixed point equations \eqref{sq1} and \eqref{sq2}, selecting $S(Q^*)$ satisfied $S(Q^*)\geq S_{\min}(Q^*)$ where 
\begin{equation} \label{eq:smin}
S_{\min}(Q^*) = \max \left\{
\frac{B_{\max} - \frac{Q^*}{\alpha}}{Q^*}, \;
\frac{\frac{Q^*}{\alpha}}{Q_{\max} - Q^*}
\right\}.
\end{equation}

\begin{prop}\cite{diagne2015state}\label{lemB}
    For any setpoint $Q^* \in \left[0, \min \left\{Q_{\max}, \mu \right\} \right]$ and for any chosen setpoint slope $S(Q^*) \in \mathbb{R}$ satisfying $S(Q^*) \ge S_{\min}(Q^*)$, where $S_{\min}(Q^*)$ is given by \eqref{eq:smin}, taking the control gains $(\Lambda_\ell, \Lambda_r)$ as solutions of
\begin{equation} \label{eq:lambda_left}
\Lambda_\ell (B_{\max} - \frac{Q^*}{\alpha}) - S(Q^*) \left(1 - e^{-\Lambda_\ell Q^*} \right) = 0,
\end{equation}
\begin{equation} \label{eq:lambda_right}
\Lambda_r \frac{Q^*}{\alpha} - S(Q^*) \left(1 - e^{-\Lambda_r (Q_{\max} - Q^*)} \right) = 0,
\end{equation}
the closed-loop system consisting of \eqref{Mp5} with an initial condition $Q_0 \in \left[0, \min \left\{Q_{\max}, \mu \right\} \right]$  and 
control law \eqref{bangbang} is asymptotically stable.
\end{prop} 

We refer the reader to \cite{diagne2015state} for the proof of Proposition \ref{lemB} and detailed derivation  of the above result.

From \eqref{MF1}--\eqref{MF7}, the predictor feedback control law of the production line is written as 
\begin{align}
& U(t)=\frac{(1+p_3(D,t))\, u(Q^*,p_1(D,t))}{(1+p_3(0,t)) K(D,D,t)} \notag  \\
    &~~~~~~~~~~~~~~~~~  -\int_0^D \frac{AP \,y\, p_2(y,t)(1+p_3(y,t))}{K(D,y,t)(1+p_3(0,t))}\,   \label{C0}   dy, \\
	&p_1(x,t)\! =\! Q(t)\!\! \notag \\
    &~+\!\! \int_0^x \!\!
 \!\!\left(\alpha\, p_2(y,t)\! - P\,(1 \! +p_3(y,t)) \,\min(p_1(y,t),\mu)  \right) dy, \label{C1}   \\  
	&p_2(x,t) = K(x,x,t) \cdot \rho(x,t) 
	 \notag \\  
	&+ \int_0^x\!\! \frac{K(x,x,t)}{K(x,y,t)}A(D\!-x\!+y)\,p_2(y,t) \, P(1 \!+p_3(y,t))  dy,  \label{C2}\\  
	&p_3(x,t) = \int_{t-\tau}^t \gamma(s, \sigma(x,t))\, Q(s)\, ds 
	 \notag \\
	& + \int_0^x \gamma(\sigma(y,t), \sigma(x,t)) \cdot P\, p_1(y,t)\,(1+p_3(y,t)) \, dy, \label{C3}\\  
	&\sigma(x,t) = t + \int_0^x P(1+p_3(y,t)) \, dy, \notag 
\end{align}
where the kernel is 
\begin{align*}
	K(x,y,t) := \exp\left(- \int_0^{y} C(x - z)\,P\,(1+p_3(z,t)) \, dz \right).
\end{align*}

{\color{black}
As stated in the following proposition, we derive a locally safe softened ``bang-bang" control law  for the closed-loop system. Specifically we ensure that the  control input $U(t) \ge 0$ and the states $(Q, \rho) \ge 0$ by deriving relevant constraints.
\begin{prop}
[Locally safe ``bang-bang" control]\label{r16}
Let the recycle rate satisfies $A>0$, the spatial friction satisfies $C(x)\ge 0$, and the initial buffer level satisfies $0 \le Q(0) < Q_{\max}$. Assume that the initial state verifies $0 \le p_1(D,0) \le Q_{\max}$ and the compatibility condition \eqref{nu0}, which ensures $U(0)>0$.

If, in addition, the following constraint holds:
\begin{equation}\label{Abound}
 \frac{M}{\underline{u}} < \frac{2}{A D P^2 (1+\tau Q_{\max})},
\end{equation}
where,
\begin{align*}
\underline{u} = \min\left\{\frac{Q^*}{\alpha},\, u(Q^*, Q(0)),\, u(Q^*, p_1(D,0))\right\}
\end{align*}
denotes the global minimum nominal control, and the constant $M$ is defined by
\begin{align}
    M = \overline{\rho}\, e^{2 A P D^2 (1+\tau Q_{\max})}, \label{Mchar}
\end{align}
with
\begin{align*}
\overline{\rho} = \max\left\{ \max_{0 \le x \le D} \rho(x,0),\; P(1+\tau Q_{\max})^2 B_{\max} \right\},
\end{align*}
then the control input $U(t)$ remains positive for all $t \ge 0$. Consequently, the solution $(Q, \rho)$ to the closed-loop system \eqref{eq buffer1}--\eqref{Mp4} remains non-negative for all $t \ge 0$. 
\end{prop}
}
\begin{proof}
    \textcolor{black}{The estimates \eqref{Abound}--\eqref{Mchar} are derived from the predictive nature of the control law \eqref{rp1}--\eqref{rp3} and \eqref{C0}, \eqref{C2} and the positivity of the actuator and plant states can be established via the method of characteristics \cite{LiYu}. We omit the details here. In contrast to \cite{diagne2015state} and \cite{diagne2017compensation}, which establish globally stabilizing, safe softened bang–bang control under a predictor feedback design, the present setting admits only local guarantees due to the presence of source terms.}
\end{proof}
\section{Simulation Results}\label{section 5}

To evaluate the effectiveness of the proposed control strategy, we carry out numerical simulations with the objective of stabilizing the queue state $Q$ at a desired setpoint $Q^* = 0.3$. The system parameters are selected as follows: the output capacity is set to $\mu = 0.8$, the processing time to $P = 0.25$, and the buffer capacity to $Q_{\max} = 1$. The maximum control input is limited by $B_{\max} = 1.2$, and the distribution coefficient is fixed as $\alpha = 0.5$. Additionally, the integration window of length $\tau = 0.2$ in the propagation speed. 

The production line domain spans from $x = 0$ to $x = 2$, where the raw materials enter at the right boundary. The system is initialized with an empty line, i.e., the part density satisfies $\rho_0(x) = 0$ for all $x$, and the buffer is also initialized to be empty over the initial integration window, that is, $Q(s) = 0$ for $s \in [-0.2, 0]$.
To capture loss and recycling effects, the dissipation coefficient is set to $C(x) \equiv 1$ (i.e., spatially uniform friction), and the total rework injection rate is set to $A = 0.1$.

To implement the delay-compensated bang–bang control, we set the slope parameter to exceed its minimum admissible value by a fixed offset: 
$S(Q^*) = S_{\min}(Q^*) + 20$. 

{\color{black}The closed-loop PDE--ODE system is discretized using a first-order upwind finite difference scheme on a uniform spatial grid of $N = 80$ points over $[0, D]$. The time integration employs the forward Euler method with a CFL-constrained time step of $\Delta t \approx 0.0025$\,s. At each time step, the multi-layer predictor integral equations are solved by forward spatial marching from $x = 0$ to $x = D$.}

The simulation results presented in Figure~\ref{odestate} compellingly demonstrate the efficacy of the proposed predictor-feedback controller by contrasting its performance with open-loop and uncompensated scenarios. In the open-loop case, the plant state $Q$ exhibits a very slow convergence towards the set-point. More critically, when the nominal bang-bang controller is applied directly without compensation, the inherent control lag leads to pronounced oscillations in both the plant state and the control input, preventing stabilization. In stark contrast, the introduction of our delay compensator yields a perfect convergence. As shown by the solid red line in Figure~\ref{odestate}, the compensated system achieves smooth and rapid stabilization at the set-point $Q^*$ without overshoot or oscillatory behavior, which means the multi-layer predictor feedback law completely compensate the delay and eliminate the influence of the source term.

\begin{figure}[t]
  \centering
 \hspace{-.4in} \includegraphics[width=1\linewidth]{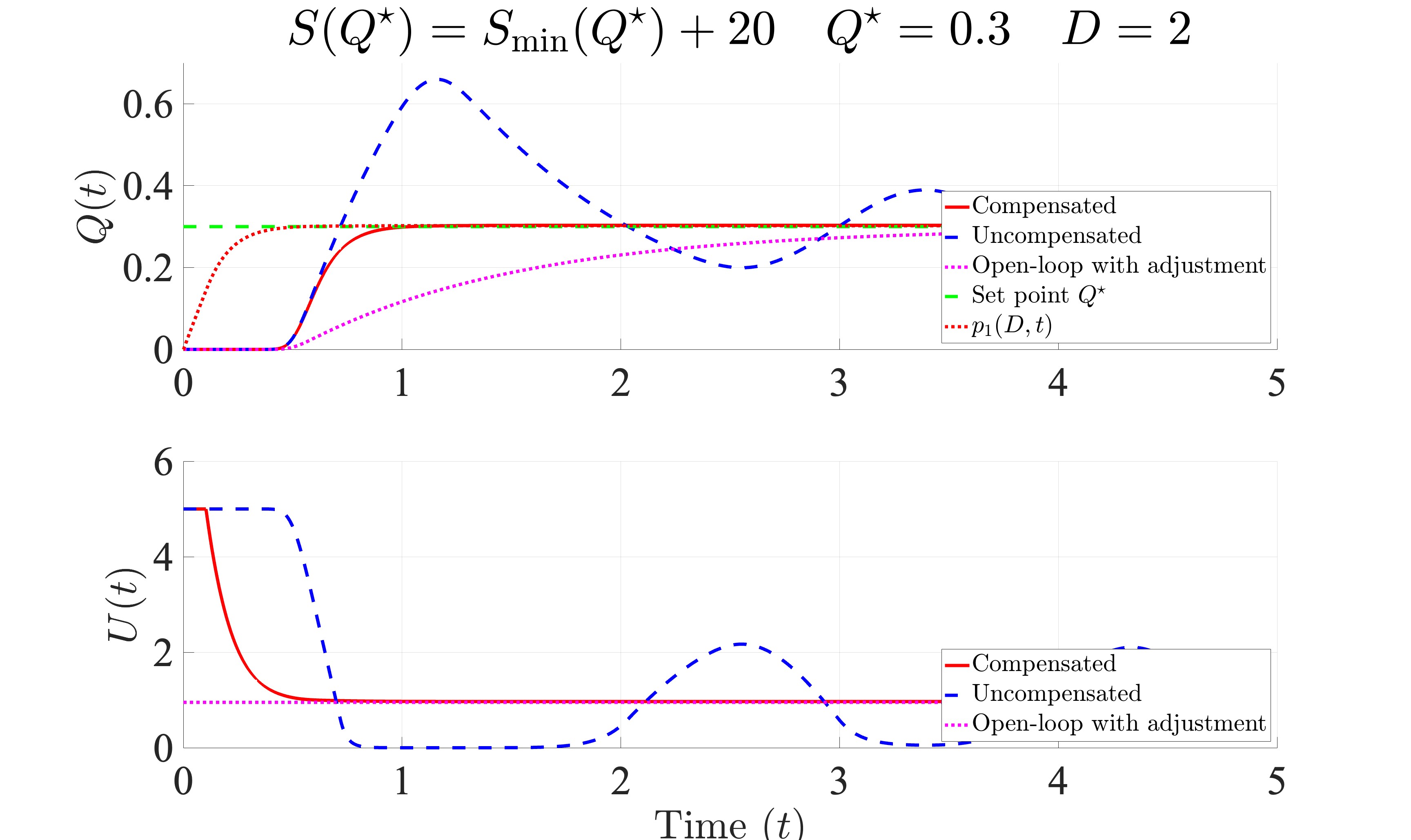}
  \caption{Evolution of ODE state under the compensated control, uncompensated control, open-loop control. }\label{odestate}
\end{figure}

\begin{figure*}[t]
  \centering
\includegraphics[width=0.8\linewidth]{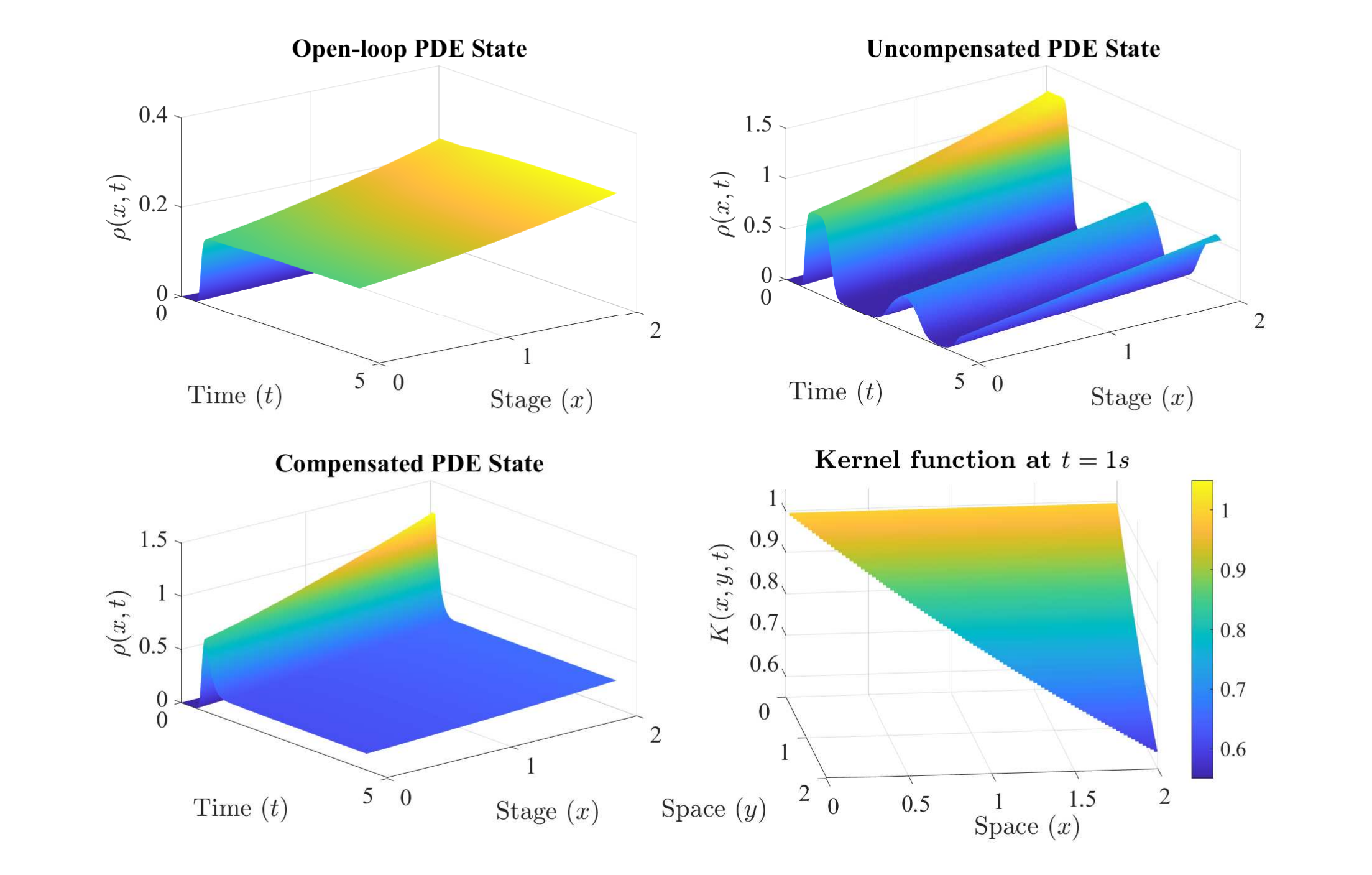}
  \caption{Evolution of PDE state under open-loop, compensated, uncompensated control and snapshot of the gain $K(x,y,t)$ }\label{pdestate}
\end{figure*}

\begin{figure}[t]
  \centering
\includegraphics[width=0.8\linewidth]{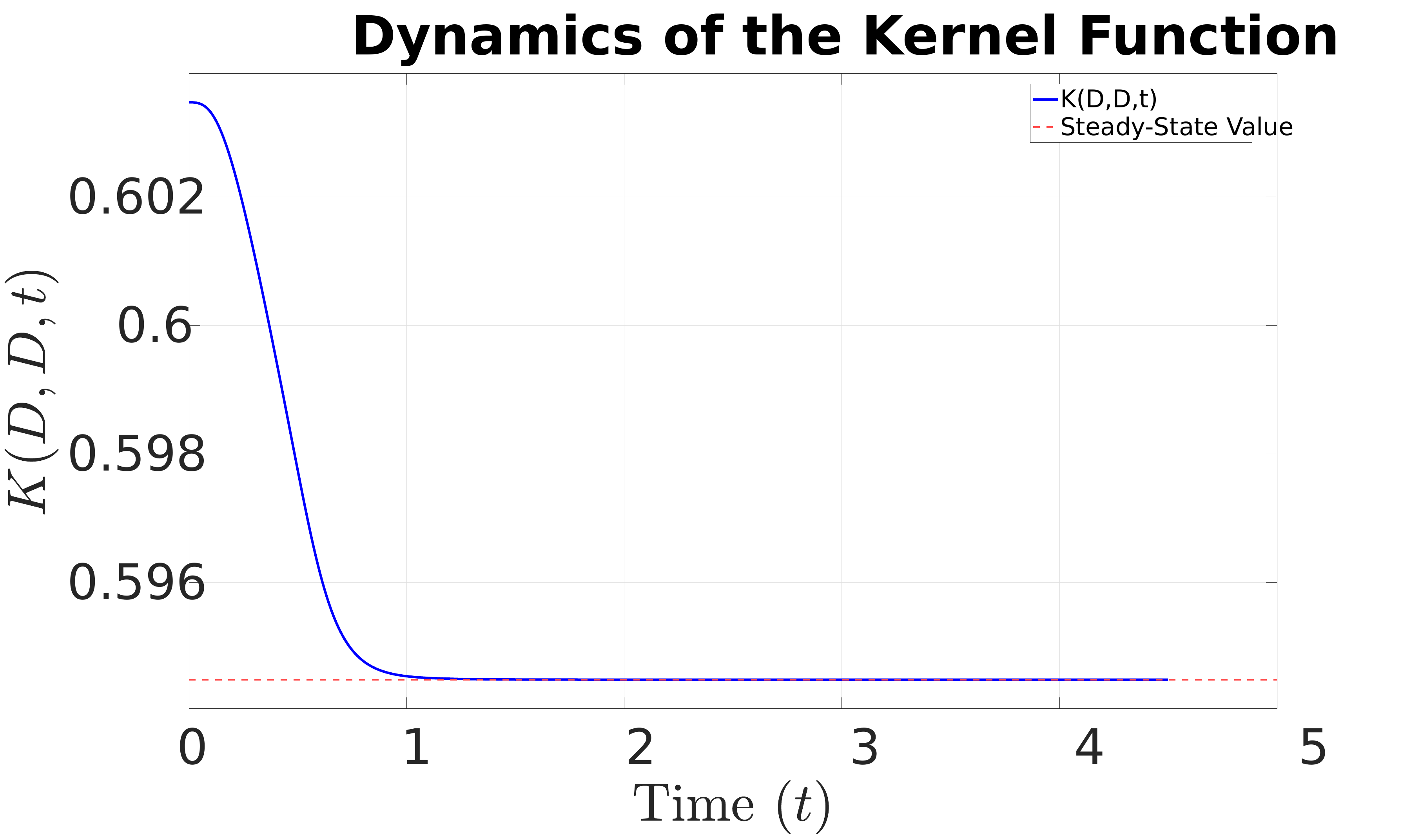}
  \caption{Evolution of  the gain $K(x,y,t)$ }\label{gainfig}
\end{figure}

{\color{black}Figure~\ref{pdestate} depicts the spatiotemporal evolution of the PDE state $\rho(x,t)$ under each scenario. In the open-loop and uncompensated cases, the density profile exhibits persistent transients and oscillatory behavior, respectively. In the compensated case, $\rho(x,t)$ converges smoothly to its steady-state profile. The fourth panel shows a snapshot of the backstepping kernel $K(x,y,t)$, which captures the combined effect of friction and recycling compensation. Figure~\ref{gainfig} shows the time evolution of $K(D,D,t)$, confirming its monotonic convergence to the steady-state value as $Q(t) \to Q^*$.}

\section{Conclusion}  \label{section con}
		In this study, we developed a control framework for coupled non-conservative PDE–ODE systems in which the propagation speed depends on the integral of past ODE states. Our approach introduces a layered predictor structure for nonlinear integro-differential equations with state-dependent input delays, modeled as a nonlinear composite PDE–ODE system. By incorporating flux-based actuation and sensing, the boundary conditions become highly nonlinear, posing challenges beyond those addressed in prior work. The control design relies on nonlinear backstepping transformations, and both global asymptotic stability (in the supremum norm of the state and actuator) and well-posedness are established through the characteristic method and a fixed-point argument, which differs from the classical approach relying on a Lyapunov argument. The proposed design applies to systems with both constant and time-varying transport speeds and retains a separation principle, relying on the design of a nominal control law for the delay-free plant. Future work will extend this framework to other classes of hyperbolic PDEs and to networks of coupled PDE–ODE subsystems.

\appendices
\section{The proofs of the lemmas in Section \ref{sec3}} \label{Appendix A}    

\subsection{Proof of Lemma \ref{lemma 9}}\label{proof 9} 
Define $A(y,t)=\lambda(p_3(y,t)) p_2(y,t)$, differentiating \eqref{MF3} w.r.t $t$, we have 
  \begin{align}\label{bp1t}
    &\partial_t p_1(x,t)=\int_{0}^{x} \frac{f_1(p_1(y,t),A(y,t) )\, \partial_t p_1(y,t)}{\lambda(p_3(y,t))}  dy  \notag \\
    &~~~+ \int_{0}^{x} \frac{f_2(p_1(y,t), A(y,t)  )\,  \partial_t A(y,t)}{\lambda(p_3(y,t))} dy \notag \\ 
    &~~~-\int_{0}^{x} \frac{\nabla \lambda(p_3(y,t)) \cdot \partial_t p_3(y,t)}{\lambda^2(p_3(y,t))}   f(p_1(y,t),A(y,t)) dy \notag \\
    &~~~+f(X(t),\,\lambda(p_3(0,t)\textcolor{black}{)} \, u(0,t)\, ),
  \end{align}
 {\color{black} differentiating \eqref{MF3} w.r.t $x$ and employing the \textbf{Newton-Leibniz formula} gives }
 \begin{align}\label{bp1x}
    &\partial_x p_1(x,t)=\int_{0}^{x} \frac{f_1(p_1(y,t),A(y,t))\,   \partial_y p_1(y,t)}{\lambda(p_3(y,t))}  dy \notag \\
    &~~~~~~+ \int_{0}^{x} \frac{f_2(p_1(y,t),A(y,t))\,  \partial_y A(y,t)}{\lambda(p_3(y,t))} dy \notag \\
    &~~~~~~-\int_{0}^{x} \frac{\nabla \lambda(p_3(y,t))\cdot \, \partial_y p_3 (y,t)}{\lambda^2(p_3(y,t))}  f(p_1(y,t),\,A(y,t)) dy \notag \\
    &~~~~~~+\frac{1}{\lambda(p_3(0,t)\textcolor{black}{)}}f(p_1(0,t),A(0,t)).
  \end{align}
Here, $f_1$ and $f_2$ represent partial derivatives of the vector function $f$ w.r.t. its first and second variable respectively.

Similarly, for simplification, we set:
\begin{align}
H(x,y,t) &= \frac{K(x,x,t)}{K(x,y,t)} = \exp\left(\int_{y}^{x}\frac{c(x-z)}{\lambda(p_{3}(z,t))}dz\right) ,\label{p2t1}\\
J(x,y,t) &= \frac{g(x-y,p_{2}(y,t))}{\lambda(p_{3}(y,t))} \label{p2t2}.
\end{align}

Thus, from \eqref{MF4},
\begin{equation}\label{p2app}
    p_2(x,t) = K(x,x,t)u(x,t) + \int_{0}^{x} H(x,y,t) J(x,y,t) dy.
\end{equation}

Differentiating \eqref{p2app} with respect to $t$ and $x$ involves careful use of the Leibniz integral rule, along with the differentiation of the kernel $K(x,x,t)$ and the auxiliary functions  
$H(x,y,t)$, $J(x,y,t)$  defined in \eqref{MF7} and \eqref{p2t1}, \eqref{p2t2}. After algebraic simplifications, we obtain the final expressions  as

\begin{align}\label{bp2t}
&\partial_t p_2(x,t)= K(x,x,t)  \,\,\partial_t u(x,t)\nonumber \\
&~~~ -\!\! K(x,x,t)u(x,t) \!\! \int_{0}^{x} c(x-z) \frac{\nabla\lambda(p_3(z,t)) \cdot \partial_t p_3(z,t)}{\lambda^2(p_3(z,t))} dz \!\! \!\! \nonumber \\
&~~~- \int_{0}^{x}  \bigg[ H(x,y,t) J(x,y,t)   \notag \\
&~~~~~~~~~~~~~~\times \int_y^x c(x-z) \, \frac{\nabla\lambda(p_3(z,t)) \cdot \partial_t p_3(z,t)}{\lambda^2(p_3(z,t))} dz \bigg] \,dy \nonumber \\
&~~~ + \int_{0}^{x}  \frac{H(x,y,t)}{\lambda(p_3(y,t))}  g_2(x-y,p_2(y,t)) \,  \partial_t p_2(y,t) \,dy \nonumber \\
&~~~-\int_{0}^{x}\!\! H(x,y,t)  \frac{g(x-y,\,p_{2}(y,t))}{\lambda^2(p_3(y,t))} \nabla\lambda(p_3(y,t)) \cdot \partial_t p_3(y,t) \,dy,
\end{align}
and 
\begin{align}\label{bp2x}
&\partial_x p_2(x,t)=K(x,x,t) \, \partial_x u(x,t) + H(x,0,t)J(x,0,t) \nonumber \\ &~~ +K(x,x,t) u(x,t)\nonumber \\
& ~~~~~~\times\!\left( \frac{c(x)}{\lambda(p_{3}(0,t))} -\int_0^x \!\! c(x-z) \frac{\nabla\lambda(p_3(z,t))  \cdot\partial_z p_3(z,t)} {\lambda^2(p_{3}(z,t))} dz \!\! \right)  \nonumber \\
&~~ + \int_{0}^{x} \frac{H(x,y,t)\,g_2(x-y,p_2(y,t))}{\lambda(p_3(y,t))}  \partial_y p_2(y,t)\, dy \nonumber \\
&~~ - \int_{0}^{x} \frac{H(x,y,t)g(x-y,p_2(y,t))\, \nabla\lambda(p_3(y,t))}{\lambda^2(p_3(y,t))}  \cdot\partial_y p_3(y,t) \, dy\notag \\
&~~-\int_{0}^{x}  H(x,y,t) J(x,y,t) \nonumber \\
& ~~~~~~ \times   \int_{y}^{x}\!\!c(x-z) \frac{\nabla \lambda(p_{3}(z,t))}  {\lambda^2(p_{3}(z,t))}   \cdot\partial_z p_3(z,t)\,dz \, dy.
\end{align}
Here, $g_1$ and $g_2$ represent partial derivatives of the scalar function $g$ w.r.t. to its first and second variable respectively.
          \textcolor{black}{
        As $\sigma(x,t) \ge t$ implies that the first integral term in \eqref{MF5} vanishes effectively for $0\le s < t-\tau$, the lower limit of integral can be treated as fixed at $0$. Consequently, differentiating \eqref{MF5} with respect to $t$ and $x$ yields:}
        \begin{align}\label{bp3t}
			&\partial_t p_3(x,t)
	=\gamma(t,\sigma(x,t) ) X(t) \notag \\
			&+\int_{t-\tau}^t \gamma_2(s,\sigma(x,t) )  X(s) ds \,\,\partial_t \sigma(x,t)\notag \\ 
			&+\int_{0}^{x}\gamma_1(\sigma(y,t),\sigma(x,t) )\,\partial_t \sigma(y,t)\frac{p_1(y,t)}{\lambda(p_3(y,t))}  dy \notag \\
			& +\int_{0}^{x}\gamma_2(\sigma(y,t),\sigma(x,t) )\,\partial_t \sigma(x,t) \frac{p_1(y,t)}{\lambda(p_3(y,t))}  dy \notag\\
			&
			+ \int_{0}^{x}\gamma(\sigma(y,t),\sigma(x,t)) \cdot \frac{\partial_t p_1(y,t)}{\lambda(p_3(y,t))} \, dy  \notag \\
			&-\int_{0}^{x} \gamma(\sigma(y,t),\sigma(x,t) )\frac{\nabla \lambda(p_3(y,t))\cdot \partial_t p_3(y,t)}{\lambda^2(p_3(y,t) )}p_1(y,t) dy \notag\\	
		\end{align} 
		and 
	\begin{align}\label{bp3x}
	&\partial_x p_3(x,t) 
	=\gamma(t,\sigma(x,t) ) \frac{p_1(0,t)}{\lambda(p_3(0,t))} \notag \\
	&+\int_{t-\tau}^t \gamma_2(s,\sigma(x,t) )  X(s) ds \,\,\partial_x \sigma(x,t) \notag \\ 
	&+\int_{0}^{x}\gamma_1(\sigma(y,t),\sigma(x,t)) \partial_y \sigma(y,t) \frac{p_1(y,t)}{\lambda(p_3(y,t))}  dy \notag \\
	& +\int_{0}^{x}\gamma_2(\sigma(y,t),\sigma(x,t)) \sigma_{\textcolor{black}{x}}(x,t) \frac{p_1(y,t)}{\lambda(p_3(y,t))}  dy \notag\\
	&
	+ \int_{0}^{x}\gamma(\sigma(y,t),\sigma(x,t)) \frac{\partial_y p_1(y,t)}{\lambda(p_3(y,t))} \,dy \notag \\
	&-\int_{0}^{x} \gamma(\sigma(y,t),\sigma(x,t))  \frac{\nabla \lambda(p_3(y,t))\cdot \partial_y p_3(y,t)}{\lambda^2(p_3(y,t) )}p_1(y,t) dy.
\end{align} 
Differentiating \eqref{MF6} at $t$ and $x$ gives
\begin{align}\label{bp4t}
    &\partial_t \sigma(x,t)=1-\int_{0}^{x} \frac{\nabla \lambda(p_3(y,t)) \cdot \partial_t p_3(y,t)}{\lambda^2(p_3(y,t))} \, dy ,
  \end{align}
  \begin{align}\label{bp4x}
    &\partial_x \sigma(x,t)=\frac{1}{\lambda(p_3(0,t))}\!- \!\! \int_{0}^{x} \frac{\nabla \lambda(p_3(y,t)) \cdot \partial_y p_3(y,t) }{\lambda^2(p_3(y,t))} \, dy .
  \end{align}
  We define the vectors $\mathbf{Z}(x,t)$ and $\mathbf{p}(x,t)$ as follows:
\begin{equation*}
  \mathbf{Z}(x,t) = \left( Z_1(x,t), Z_2(x,t),Z_3(x,t), Z_4(x,t) \right)^{\top}\end{equation*}

\begin{equation*}\mathbf{p}(x,t) = \left(p_1(x,t),p_2(x,t),p_3(x,t),\sigma(x,t) \right)^{\top},  
\end{equation*}
where
\begin{equation}
    \mathbf{Z}(x,t) = \partial_t \mathbf{p}(x,t) - \lambda \left( \int_{t-\tau}^{t} X(s) \, ds \right) \partial_x \mathbf{p}(x,t).
\end{equation}
		Combining \eqref{M3}, \eqref{bp1t}, \eqref{bp2x},  
		and considering \eqref{MF3}--\eqref{MF5} at $x=0$, note that $A(y,t)=\lambda(p_3(y,t)) p_2(y,t)$, we have
\begin{align}\label{Z1}
    &Z_1(x,t)=\int_{0}^{x} \frac{f_1(p_1(y,t), A(y,t) )}{\lambda(p_3(y,t))} \, Z_1(y,t) \,dy   \notag \\
    &~+ \int_{0}^{x} f_2(p_1(y,t), A(y,t) )   \,Z_2(y,t) \,dy  \notag \\ 
    &~+ \int_{0}^{x} \frac{f_2(p_1(y,t), A(y,t)   ) \, \nabla \lambda(p_3(y,t))\, p_2(y,t)}{\lambda(p_3(y,t))} \cdot\,Z_3(y,t) \,dy  \notag \\
    &~-\int_{0}^{x} \frac{\nabla \lambda(p_3(y,t)) \cdot Z_3(y,t)}{\lambda^2(p_3(y,t))}   f(p_1(y,t),A(y,t)) dy. 
  \end{align}

 Via \eqref{bp2t} and \eqref{bp2x}, using \eqref{M3} and note that
\begin{equation*}
    H(x,0,t)J(x,0,t)= \frac{K(x,x,t)g(x,u(0,t))}{\lambda(p_3(0,t))},
\end{equation*} 
we have 
\begin{align}
&Z_2(x,t)= -K(x,x,t) u(x,t) \notag \\
&~~~~~~~~~~~~~\times \int_0^x c(x-y) \frac{\nabla\lambda(p_3(y,t))\cdot  Z_3(y,t) } {\lambda^2(p_{3}(y,t))}\,dy  \nonumber \\
& + \int_{0}^{x} \frac{H(x,y,t)g_2(x-y,p_2(y,t))}{\lambda(p_3(y,t))}\, Z_2(y,t) dy \nonumber \\
& - \int_{0}^{x} \frac{H(x,y,t)g(x-y,p_2(y,t)) \nabla\lambda(p_3(y,t))\cdot Z_3(y,t)}{\lambda^2(p_3(y,t))} \,  dy\notag \\
&-\int_{0}^{x}  H(x,y,t) J(x,y,t) \nonumber \\
&~~~~~~~~~~~~ \times   \int_{y}^{x}\!\! \frac{c(x-\zeta)\,\nabla \lambda(p_{3}(\zeta,t))}{\lambda^2(p_{3}(\zeta,t))}\cdot Z_3(\zeta,t) \,d\zeta \, dy.
\end{align}
	From \eqref{bp3t}--\eqref{bp4x}, we have 		\begin{align}\label{Z3}
		&Z_3(x,t)= \int_{t-\tau}^t \gamma_2(s,\sigma(x,t) ) X(s) ds\, Z_4(x,t) \notag \\ 
&+\int_{0}^{x}\gamma_1(\sigma(y,t),\sigma(x,t) )\,Z_4(y,t)  \frac{p_1(y,t)}{\lambda(p_3(y,t))} dy \notag \\
&+\int_{0}^{x}\gamma_2(\sigma(y,t),\sigma(x,t) )\,Z_4(x,t)  \frac{p_1(y,t)}{\lambda(p_3(y,t))}  dy \notag\\
		&
		+ \int_{0}^{x}\gamma(\sigma(y,t),\sigma(x,t)) \frac{Z_1(y,t)}{\lambda(p_3(y,t))} dy \notag \\
		&-\int_{0}^{x} \gamma(\sigma(y,t),\sigma(x,t)) \frac{\nabla \lambda(p_3(y,t))\cdotp Z_3(y,t)}{\lambda^2(p_3(y,t) )}p_1(y,t) dy ,	
	\end{align} 
    and 
	\begin{align}\label{Z4}
		Z_4(x,t)=-\int_{0}^{x} \frac{\nabla \lambda(p_3(y,t))\cdotp Z_3(y,t)}{\lambda^2(p_3(y,t) )} dy.
	\end{align}
\textcolor{black}{Noticing that $\mathbf{Z}(0,t)=\mathbf{0}$, we deduce from \eqref{Z1}
  to 
\eqref{Z4} that}
$\mathbf{Z}(x,t)\equiv \mathbf{0}$ for all $x\in [0,D]$, which means 
  \begin{align}\label{pis0}
  		\partial_t p_i(x,t)-\lambda \left(\int_{t-\tau}^{t} X(s) \, ds\right)\partial_x p_i(x,t)=0 
  \end{align}
  for $i=1,2,3$ and 
  \begin{equation}\label{s0}
      \partial_t \sigma(x,t)-\lambda\left(\int_{t-\tau}^{t} X(s) \, ds\right)\partial_x \sigma(x,t)=0.
  \end{equation}
 
 From \eqref{MF4} and \eqref{nonlinear backstepping}, we have $w(x,t)=\lambda(p_3(x,t)) p_2(x,t)-\kappa(p_1(x,t))$. Therefore, from  \eqref{pis0}
\begin{equation}\label{wtwxu}
\partial_t w(x,t)=\lambda\left(\int_{t-\tau}^{t}X(s)ds\right)\partial_x w(x,t).
\end{equation}
We can get \eqref{non tar 1} from \eqref{M1}--\eqref{M2}, combining \eqref{MF3}, \eqref{MF5}, \eqref{nonlinear backstepping} at $x=0$. 

 The boundary condition (\ref{non tar 3}) comes from \eqref{M4},     \eqref{MF1}  together with \eqref{nonlinear backstepping} at $x=D$.

\subsection{Proof of Lemma \ref{lemma 10}}
Comparing \eqref{npi1}--\eqref{npi4} and \eqref{MF3}--\eqref{MF6}, following an argument analogous to the proof of       Lemma~\ref{lemma 9},  we can conclude that
 \begin{align}\label{pi01}
  		\partial_t \pi_i(x,t)-\lambda \left(\int_{t-\tau}^{t} X(s) \, ds\right)\partial_x \pi_i(x,t)=0 
  \end{align}
  for $i=1,2,3$ and 
  \begin{equation}\label{pi02}
      \partial_t \overline{\sigma}(x,t)-\lambda\left(\int_{t-\tau}^{t} X(s) \, ds\right)\partial_x \overline{\sigma}(x,t)=0.
  \end{equation}
  From \eqref{kernel L}, we have
 \begin{align} \label{Ltxx}
     & \partial_t L (x,x,t)  =   L(x,x,t) 
      \!\! \int_0^x\!\!\! c(x-z) \!\! \frac{\nabla\lambda(\pi_3(z,t)) \,\, \partial_t \pi_{3}(z,t)} {\lambda^2(\pi_{3}(z,t))} dz , 
 \end{align}
 and 
 \begin{align}\label{Ltxy}
     & \partial_t L(x,y,t)  =   L(x,y,t) 
      \!\! \int_0^y\!\!\! c(x-z) \!\! \frac{\nabla\lambda(\pi_3(z,t)) \,\, \partial_t \pi_{3} (z,t)} {\lambda^2(\pi_{3}(z,t))} dz . 
 \end{align}
In this lemma, for the sake of simplification, we define 
\begin{equation}\label{G}
    G(x,y,t)=\frac{g(x-y,\pi_2(y,t))}{\lambda(\pi_3(y,t))}.
\end{equation}
 Differentiating \eqref{non inverse backstepping} at $t$, from \eqref{kernel L}, \eqref{Ltxx} and  \eqref{Ltxy},  we have 
\begin{align} \label{utinverse}
      &\partial_t u(x,t)=\frac{\partial }{\partial t} \left[ \frac{\pi_2(x,t)}{\lambda(\pi_3(x,t)\textcolor{black}{)}}\right] \notag \\
     &+ L(x,x,t) \frac{\pi_2(x,t)}{\lambda(\pi_3(x,t)\textcolor{black}{)}}
        \notag \\ & ~~~~~~~~~~~~\times   \int_0^x c(x-z) \frac{\nabla\lambda(\pi_3(z,t)) } {\lambda^2(\pi_{3}(z,t))}  \partial_t \pi_{3}(z,t) dz   \notag \\
  &-\int_0^x    L(x,y,t)G(x,y,t) \notag \\ &~~~~~~~~~~~~\times\left(  \int_0^y \!\! \frac{c(x-z)}{\lambda^2(\pi_3(z,t))}\nabla  \lambda(\pi_3 (z,t)) \, \partial_t  \pi_{3}(z,t) \, dz \!\! \right) \,  dy \notag \\
  &- \int_0^y  L(x,y,t)\,   \frac{ g_2(x-y,\pi_2(y,t))}{\lambda(\pi_3(y,t))}  \partial_t \pi_{2}(y,t)  \, dy \notag\\
   &  + \int_0^y  L(x,y,t)\, \frac{ G(x,y,t)  \nabla\lambda(\pi_3(y,t))}{\lambda(\pi_3(y,t))} \partial_t  \pi_{3}(y,t)   \,  dy.
 \end{align}
 
 From \eqref{kernel L}, we can also obtain 
 \begin{align}\label{Lxxx}
     & \partial_x L(x,x,t)  =   L(x,x,t)
     \left( -\frac{c(x)}{\lambda(\pi_{3}(0,t))} \right. \notag \\ &\left. + \int_0^x c(x-z) \frac{\nabla\lambda(\pi_3(z,t)) \, \partial_z \pi_{3}(z,t)} {\lambda^2(\pi_{3}(z,t))} dz \right), 
 \end{align}
 and 
 \begin{align}\label{Lxxy}
      \partial_x L(x,y,t) =  - L(x,y,t)
    \int_{0}^{y}\frac{c'(x-z)}{\lambda(\pi_{3}(z,t))}dz.
 \end{align}
 
Differentiating \eqref{non inverse backstepping} at $x$, from \eqref{kernel L}, \eqref{Lxxx} and \eqref{Lxxy}, we have 
\begin{align} \label{uxinverse}
     & \partial_x u(x,t)= \frac{\partial }{\partial x} \left[ \frac{\pi_2(x,t)}{\lambda(\pi_3(x,t)\textcolor{black}{)}}\right] -L(x,0,t)G(x,0,t) \notag \\
     &~~~+L(x,x,t) \frac{\pi_2(x,t)}{\lambda(\pi_3(x,t)\textcolor{black}{)}}
     \left( -\frac{c(x)}{\lambda(\pi_{3}(0,t))} \right. \notag \\ &\left. ~~~~~~~~~~~~~~~~~~~~+ \int_0^x c(x-z) \frac{\nabla\lambda(\pi_3(z,t)) } {\lambda^2(\pi_{3}(z,t))}  \, \partial_z \pi_{3}(z,t) dz \!\! \right)  \notag \\
  &~~~-\int_0^x    L(x,y,t)G(x,y,t) \left( -\frac{c(x)}{\lambda(\pi_3(0,t))} \right.  \notag \\
  &~~~~~~~~~~~+\left. \int_0^y \!\! \frac{c(x-z)}{\lambda^2(\pi_3(z,t))}\nabla  \lambda(\pi_3 (z,t)) \,  \partial_z \pi_{3}(z,t)\, dz \right) \,  dy \notag \\
  &~~~- \int_0^y  L(x,y,t)\,  \frac{ g_2(x-y,\pi_2(y,t))}{\lambda(\pi_3(y,t))} \,  \partial_y \pi_{2}(y,t) \, dy \notag\\
     &~~~+ \int_0^y  L(x,y,t) \frac{ G(x,y,t)  \nabla\lambda(\pi_3(y,t))}{\lambda(\pi_3(y,t))}\,  \partial_y \pi_{3}(y,t) \, dy.
 \end{align}
 From \eqref{pi01}, \eqref{pi02}, \eqref{utinverse}, \eqref{uxinverse}, we have 
 \begin{align}\label{ulemm10}
 &\partial_t u(x,t)-\lambda\left(\!\!\int_{t-\tau}^{t} X(s) \, ds\!\!\right) u_x(x,t)   \notag\\
&=
L(x,x,t)c(x) \frac{\pi_2(x,t)}{\lambda(\pi_3(x,t)\textcolor{black}{)}}-c(x) \int_0^x    L(x,y,t)G(x,y,t) dy  \notag \\
&~~~~~~~~~~~~+\lambda\left(\int_{t-\tau}^{t} X(s) \, ds\right) L(x,0,t) G(x,0,t). 
 \end{align}
 Substituting \eqref{non inverse backstepping}, \eqref{npi2}, \eqref{npi3}, \eqref{kernel L}, \eqref{G} into \eqref{ulemm10}, we obtain \eqref{M3}. Substituting  $x=0$ into \eqref{non tar 1}, \eqref{non inverse backstepping}, \eqref{npi1}, \eqref{npi3}, we have \eqref{M1}--\eqref{M2}. 
Finally, we verify that the boundary condition \eqref{M4} is satisfied. To do this, we first establish the equivalence 
\begin{equation}\label{pp}
    p_i\equiv \pi_i
\end{equation}
for $i=1,2,3$ and 
\begin{equation}\label{ssi}
    \sigma\equiv \overline{\sigma}.
\end{equation}
Indeed, from \eqref{pis0}, \eqref{s0} and \eqref{pi01}, \eqref{pi02}, they satisfy the same transport equation. Moreover, evaluating $x=0$ into \eqref{MF3}--\eqref{MF6}, \eqref{npi1}--\eqref{npi4} and \eqref{non inverse backstepping}, we can conclude that $p_i(0,t)=\pi_i(0,t)$ and $\sigma(0,t)=\overline{\sigma}(0,t)$. From the uniqueness  of the solutions to the transport equations, we obtain \eqref{pp} and \eqref{ssi}.

After that, by substituting 
$x=D$
 into equation \eqref{non inverse backstepping}   and utilizing equation \eqref{MF1}, \eqref{MF4} at $x=0$, \eqref{MF7},  \eqref{non tar 3}, \eqref{kernel L}, 
 as well as 
 \eqref{pp}--\eqref{ssi}, we can conclude \eqref{M4}.

 \subsection{Proof of Lemma \ref{lemma 11}}
 

We first establish a bound for the predictor state $p_2$. 
\textcolor{black}{Taking $t$ as a parameter, using \eqref{eqlam} and \eqref{MF7}, for any fixed $t \geq 0$, we define
\begin{align}
    K_1 &= \max_{\substack{0 \leq x \leq D}} |K(x,x,t)|, \label{K1} \\
    K_2 &= \max_{\substack{0 \leq x \leq D}} \max_{\substack{0 \leq y \leq x}} \left| \frac{K(x,x,t)}{K(x,y,t)} \right|. \label{K2}
\end{align}
 }
From \eqref{MF4} and using \textcolor{black}{\eqref{eqlam}}, \eqref{K1}, \eqref{K2} we derive 
\begin{align}\label{p2ie}
   &|p_2(x,t)| \leq K_1\,|u(x,t)| + \int_0^x \frac{K_2}{\epsilon} \,|g(x-y,\,p_2(y,t))| \, dy. 
\end{align}

Applying Assumption \ref{assump3} to \eqref{p2ie} yields:
\begin{align}\label{p2158}
   |p_2(x,t)|
           &  \leq  K_1\,|u(x,t)|   + \frac{K_2\, L_{g}}{\epsilon} \int_0^x |p_2(y,t)| \, dy. 
\end{align}
 By Gronwall's inequality,  an upper bound for $\sup_{x \in [0,D]} |p_2(x,t)|$ is given by:
\begin{align}\label{p2norm}
   &\sup_{x \in [0,D]} |p_2(x,t)|  \leq   K_1\!\! \sup_{x \in [0,D]} |u(x,t)|   e^{\frac{K_2 L_{g} D}{\epsilon}}  .
\end{align}

After the derivation of the uniform estimate of $p_2$, we move to the deduction of the bound of $p_1$.

	Taking the derivative of \textcolor{black}{\eqref{MF3}} w.r.t. $x$, we get
	\begin{equation}
		\partial_x p_1(x,t)=\frac{f(p_1(x,t),\, \textcolor{black}{\lambda(p_3(x,t))} \, p_2(x,t))}{\lambda(p_3(x,t))}  \label{lemma4 1},
	\end{equation}
	with the boundary condition
	\begin{equation}
		p_1(0,t)=X(t). \label{lemma4 2}
	\end{equation}
	From Assumption \ref{Assumption 1}, 
	noticing \textcolor{black}{\eqref{eqlam}} and (\ref{equation 8}), we have
	\begin{align*}
		&\frac{\partial \Theta(p_1(x,t))}{\partial p_1}  \frac{f(p_1(x,t),\, \textcolor{black}{\lambda(p_3(x,t))} \, p_2(x,t))}{\lambda(p_3(x,t))}  \notag \\
		&\leq \frac{1}{\lambda(p_3(x,t) )} \bigg(\Theta(p_1(x,t))+\mathcal{G}_3(\textcolor{black}{\lambda(p_3(x,t))} |p_2(x,t)|)\bigg).
	\end{align*}
	\textcolor{black}{From the property of $\mathcal{K}_{\infty}$ function $\mathcal{G}_3$ and \eqref{eqlam}, there exists constant $\lambda_1$}, which together with (\ref{lemma4 1}) gives
	\begin{equation}
		\frac{\partial \Theta(p_1(x,t))}{\partial x} \leq \frac{1}{\epsilon} \Theta(p_1(x,t))+\frac{\textcolor{black}{\lambda_1}}{\epsilon}\mathcal{G}_3(|p_2(x,t)|) . \label{lemma4 4}
	\end{equation}
	Thus,
	\begin{equation}
		\Theta(p_1(x,t))\leq e^{x/\epsilon}\Theta(p_1(0,t))+\frac{\textcolor{black}{\lambda_1}}{\epsilon}\int_{0}^{x}e^{\frac{x-y}{\epsilon}}\mathcal{G}_3|p_2(y,t)|)dy.
	\end{equation}
	We derive from (\ref{lemma4 2}) that
	\begin{align}
    \Theta(p_1(x,t))&\leq e^{D/\epsilon}\Theta(X(t)) \notag \\
    &+\textcolor{black}{\lambda_1} \left(e^{D/\epsilon}-1\right) \mathcal{G}_3\left(\sup_{x \in [0 ,D]}|p_2(x,t)|\right). 
	\end{align}

	Using (\ref{equation 7}),  the following inequality holds
	\begin{align}\label{p1norm}
		\sup_{x \in [0 ,D]}&|p_1(x,t)|\leq   \mathcal{G}^{-1}_1 \Bigg(e^{D/\epsilon} \mathcal{G}_2(|X(t)|)\notag \\
		&+\textcolor{black}{\lambda_1} \left(e^{D/\epsilon}-1\right)\mathcal{G}_3\left(\sup_{x \in [0 ,D]}|p_2(x,t)| \right)\Bigg)
	\end{align}
	for all $x\in [0,D]$. 
    
    The continuity assumption on $\kappa$ in Assumption \ref{Assumption 2} implies the existence of a class $\mathcal{K}_{\infty}$ function $\mathcal{\rho}$ such that for all $\xi\in \mathbb{R}^n$
	\begin{equation}\label{ieq rho}
		|\kappa(\xi)|\leq \mathcal{\rho}(|\xi|).
	\end{equation}
    
	From \eqref{eqlam}, \eqref{K1}, \eqref{K2}, Assumption~\ref{assump3}, backstepping transformation \eqref{nonlinear backstepping} and  \eqref{ieq rho}, we have 
	\begin{align}\label{wunorm}
	&	\sup_{x \in [0 ,D]}|w(x,t)|\leq  \frac{K_1}{\epsilon}\,\sup_{x \in [0 ,D]}|u(x,t)| + \rho\left(\sup_{x \in [0 ,D]}|p_1(x,t)|\right) \notag \\
        &~~~~~~~~~~~~~+\frac{\overline{\lambda} \, K_2}{\epsilon} \left( L_g\,  \sup_{x \in [0 ,D]}|p_2(x,t)| \,\right).
	\end{align}
    Substituting \eqref{p2norm}, \eqref{p1norm} into \eqref{wunorm}, from the property of $\mathcal{K}_\infty$ function, we can deduce that there exists a $\mathcal{K}_\infty $ function {$\mathcal{K}_3$}, thereby establishing Lemma~\ref{lemma 11}.

  \subsection{Proof of Lemma \ref{lemma 12}}
  The proof consists of two main steps. First, we establish a bound for the predictor state $\sup_{x\in [0,D]}|\pi_1(x,t)|$. Second, we use this bound to derive the final estimate for $\sup_{x\in [0,D]}|u(x,t)|$. In this section, we treat the variable $t$ as a parameter.

Differentiating (\ref{npi1}) with respect to $x$ and substituting $(\ref{npi2})$ yields the dynamics of $\pi_1$ for $x \in [0,D]$:
\begin{align}\label{lemma13_1}
 \partial_x \pi_{1} = \frac{f(\pi_1,  \kappa(\pi_1) + w) )} {\lambda(\pi_3)}, \quad \pi_1(0,t)=X(t).
\end{align}
 We proceed by using the forward completeness property from Assumption~\ref{Assumption 1}.  Let $V(x) := \Theta(\pi_1(x,t))$, where $\Theta$ is the smooth positive definite function from $(\ref{equation 7})$--$(\ref{equation 8})$.
Letting $\omega_{\rm arg} = \kappa(\pi_1) + w$ and differentiating $V$ with respect to $x$ along the solution of $(\ref{lemma13_1})$ gives
\begin{align*}
\frac{d V}{d x} &= \frac{\partial \Theta(\pi_1)}{\partial \pi_1} \partial_x \pi_1 = \frac{1}{\lambda(\pi_3)} \frac{\partial \Theta(\pi_1)}{\partial \pi_1} f(\pi_1, \omega_{\rm arg}) \\
&\leq \frac{1}{\lambda(\pi_3)} \left( \Theta(\pi_1) + \mathcal{G}_3(|\omega_{\rm arg}|) \right).
\end{align*}
Using the bounds from \textcolor{black}{$(\ref{eqlam})$} and $(\ref{ieq rho})$ yields
\begin{align*}
|\omega_{\rm arg}| \leq  \rho(|\pi_1|) + \sup_{x\in[0,D]}|w(x,t)| .
\end{align*}
Substituting this into the inequality for $V$ and using \textcolor{black}{$(\ref{eqlam})$}, we obtain
\begin{align}
\frac{d V}{d x} \leq \frac{1}{\epsilon} \left( V + \mathcal{G}_3\left(\rho(|\pi_1|) +  \sup_{x\in[0,D]}|w(x,t)|\right) \right).
\end{align}
From $(\ref{equation 7})$, we have $|\pi_1| \leq \mathcal{G}_1^{-1}(V)$. Let $\tilde{\rho} = \rho \circ \mathcal{G}_1^{-1}$, which is a class $\mathcal{K}_\infty$ function. Using the property $\mathcal{G}_3(a+b)\leq \mathcal{G}_4(a) + \mathcal{G}_4(b)$ {(e.g. $\mathcal{G}_4(s)=\mathcal{G}_3(2s)$, we have $\mathcal{G}_3(a+b)\leq \mathcal{G}_3(2 \max\{a,b\}) \leq \mathcal{G}_3(2 a)+\mathcal{G}_3(2 b)$)} for some class $\mathcal{K}_\infty$ function $\mathcal{G}_4$, we get
\begin{align}\label{lemma13_3}
\frac{d V}{d x} \leq W(V) + C,
\end{align}
where $W(V) := \frac{1}{\epsilon}[V + \mathcal{G}_4(\tilde{\rho}(V))]$ is a class $\mathcal{K}_\infty$ function, and $C :=\frac{1}{\epsilon}\, \mathcal{G}_4(\sup_{x\in[0,D]}|w(x,t)|)$ is a non-negative constant, since we consider $t$ as
a parameter.

By the classical comparison principle, $V$ is bounded by the solution $z$ of the ODE 
\begin{equation}\label{Z(X)}
    \dot{z}(x) = W(z(x)) + C,
\end{equation}
with $z(0)=V(0)=\Theta(X(t))$. There exists a  solution $z$ on $[0,D]$ denoted by $\beta(z(0), C, x)$, where $\beta(z_0, C, D)\geq \beta(z_0, C, x)$ for all $x\in [0,D]$ since the non-negative property of the right side of \eqref{Z(X)}. 

From \eqref{Z(X)}, we can deduce that  $\mathcal{G}_5(x)=\beta(x, x, D)$ is a class $\mathcal{K}_{\infty}$ function. Specifically, $\mathcal{G}_5(0)=\beta(0,0,D)=0$ since $z\equiv 0$ is the solution for $z_0=0$, $C=0$; moreover, since $\beta$ is strictly increasing in $z_0$ and $C$ respectively from \eqref{Z(X)}, $\mathcal{G}_5$ is strictly increasing.

From \eqref{equation 7}, since $z(0) = \Theta(X(t))$ and $C$ is a class $\mathcal{K}_{\infty}$ function of $\sup_{x\in[0,D]} |w(x,t)|$, $\mathcal{G}_5$ is a composition of class $\mathcal{K}_\infty$ functions of $|X(t)|$ and $\sup_{x\in[0,D]} |w(x,t)|$. This implies the existence of a class $\mathcal{K}_\infty$ function $\mathcal{G}_6$ such that
\begin{align}\label{iepi1}
\sup_{x\in [0,D]}\big|\pi_1(x,t)\big|\leq \mathcal{G}_6 \bigg(\big|X(t)\big|+\sup_{x\in [0,D]}\big|w(x,t)\big|\bigg).
\end{align}
Since $\pi_1$ is bounded, we now estimate \textcolor{black}{$u$}.
From \eqref{eqlam}, \eqref{kernel L}, there exists a constant $\overline{L}$ such that 
\begin{equation}\label{overlineL}
    L(x,y,t) \leq \overline{L}
\end{equation}
for all $0\leq y\leq x \leq D$.
From \eqref{overlineL}, the inverse transformation \eqref{non inverse backstepping} and Assumption~\ref{assump3}, we have
\begin{align}\label{lemma13_5}
&\sup_{x\in [0,D]}\big|u(x,t)\big| \leq \frac{\overline{L}}{\epsilon}\left( \sup_{x\in [0,D]}\big|w(x,t)\big| + \rho(\sup_{x\in [0,D]}\big|\pi_1(x,t)\big|) \right) \notag \\
&~~~~~~ + \frac{D \overline{L} L_g}{\epsilon} \sup_{x \in [0,D]}|\pi_2(x,t)|.
\end{align}
 From \eqref{npi2}, 
 $\sup_{x\in[0,D]}|\pi_2(x,t)| \leq (\sup_{x\in[0,D]}|w(x,t)| + \rho(\sup_{x\in[0,D]}|\pi_1(x,t)|))/\epsilon$. Substituting this and the bound \eqref{iepi1} into $(\ref{lemma13_5})$ confirms the existence of a class $\mathcal{K}_\infty$ function $\mathcal{K}_4$ that satisfies $(\ref{norm 6})$.

\section{Well-posedness proof of Theorem~\ref{thm3}} \label{Appendix B}
\subsection{Well-posedness proof of the initial condition of \eqref{non tar 1}--\eqref{non tar 3}}

Given the initial condition $h$ and $ u_0$ in Theorem \ref{thm3}, we begin by rigorously showing that the initial condition $w_0$ for the target state is well-defined. This is achieved by proving that the initial predictor states  \eqref{MF3}--\eqref{MF7}, which defines the backstepping transformation \eqref{nonlinear backstepping} at $t=0$,  admit a unique solution $(p_{1,0},\,p_{2,0},\,p_{3,0},\,\sigma_0)(x)$.

Since the switching nature of the function $\gamma$ in \eqref{MF5} dependent on the state $\sigma_0$, we establish the proof via the method of steps.

For a sufficiently small initial interval $[0, \delta]$ where $\sigma_0< \tau$, from \eqref{W}, \eqref{MF5}, we have
\begin{equation}
    p_{3,0}(x) = \int_{\sigma_0(x)-\tau}^0  X(s)\, ds 
	 + \int_0^x  \frac{p_{1,0}(y)}{\lambda(p_{3,0}(y))}\, dy,  
\end{equation}

The initial predictors, namely  \eqref{MF3}--\eqref{MF7} at $t=0$, reduce to a standard set of coupled Volterra integral equations. On this interval, from a classical result based on the Contraction Mapping Principle for such equations, relying on the Assumption~\ref{Assumption 1}, the assumed regularity of $f$, $g$, $\lambda$, $c$ and the initial data $h$, $ u_0$, we can deduce  the existence of a unique $C^1$ solution $(p_{1,0},\,p_{2,0},\,p_{3,0},\,\sigma_0)$. 

 This local solution can be uniquely extended over the entire compact interval $[0,D]$ by a step-by-step continuation argument. As the solution is extended from a point $x_k$ to $x_{k+1}$, the unique $C^1$ solution on $[0, x_k]$ serves as a well-defined history for the next interval. During this process, as long as $\sigma_0 \leq \tau$, the system of Volterra equations maintains its initial structure.

Should the extension reach a point where $0 \leq x \leq D$ and $\sigma_0 > \tau$, the structure of the equation for $p_{3,0}$ changes. Specifically, the integral involving the ODE's history vanishes. The equation for $p_{3,0}$ then takes the form of a functional integral equation:
\begin{align}\label{eq:p30_advanced}
    p_{3,0}(x) = \int_{\sigma_0^{-1}(\sigma_0(x)-\tau)}^{x} \frac{p_{1,0}(y)}{\lambda(p_{3,0}(y))} dy.
\end{align}
This transition in the structure of the equation for $p_{3,0}$ is rigorously handled by the continuation argument. The strict positivity of the propagation speed \eqref{eqlam} ensures that $\sigma_0$ is a strictly increasing, continuously differentiable function, and thus locally invertible. Consequently, the term $\sigma_0^{-1}(\sigma_0(x) - \tau)$ is well-defined at each step, as its evaluation depends only on the solution over the preceding interval $[0,x)$, which has already been uniquely constructed. The problem can therefore be re-posed at each stage as a well-defined local system of Volterra equations, which is known to admit a unique $C^1$ solution. This constructive procedure, repeated over a finite number of steps, covers the entire compact domain $[0,D]$ and guarantees that a unique $C^1([0,D])$ solution for the initial predictor states $(p_{1,0},\, p_{2,0},\, p_{3,0},\, \sigma_0)$ exists.

The existence of this unique $C^1([0,D])$ solution for the initial predictor states directly implies that the initial condition for the target system, $w_0$, as defined by the backstepping transformation \eqref{nonlinear backstepping} at $t=0$, is itself well-defined and Lipschitz continuous on $[0,D]$. Thus there exists a Lipschitz constant $L$ such that for all $x_1,x_2\in [0,D]$
\begin{equation}
		|w_0(x_1)-w_0(x_2)|\leq L|x_1-x_2| .
	\end{equation}
\subsection{Well-posedness of the Closed-Loop System}

With a well-defined initial condition $w_0$, we prove the well-posedness of the target system \eqref{non tar 1}--\eqref{non tar 3}. The proof hinges on a fixed-point argument detailed in our prior conference paper~\cite{Li2025CPDE}, where an identical target system was analyzed.

The core of the analysis in~\cite{Li2025CPDE} is to first establish the existence of a unique solution pair $(X, \xi) \in C^1 \times C^1$, where $\xi$ is the integrated propagation speed defined as:
\begin{equation} \label{psi}
    \xi(t) := \int_{0}^{t} \lambda\left(\int_{\sigma-\tau}^{\sigma}X(\eta)d\eta\right)d\sigma.
\end{equation}
Once the unique solution for $\xi$ is secured, the solution for the distributed state $w$ is constructed via the method of characteristics. As shown in~\cite[eq.~(46)]{Li2025CPDE}, this solution is given by $w(x,t) = w_0(x+\xi(t))$, for $x+\xi(t) \in [0,D]$.

To connect this expression with the notation used in equation \eqref{wsolution}, we recall that the characteristic curve $\zeta(s;x,t)$ is defined by \eqref{cline}, with its explicit solution given by:
\begin{equation} \label{xi_zeta}
    \zeta(s;\,x,t) = x + \xi(t) - \xi(s).
\end{equation}
By setting $s=0$ in \eqref{xi_zeta} and noting that $\xi(0)=0$ by definition, we establish the identity $x+\xi(t) = \zeta(0;x,t)$. 

Since $w_0$ is Lipschitz and $\xi$ is $C^1$, it follows that the solution $w$ is locally Lipschitz continuous. With this solution established, 
the equations \eqref{npi1}--\eqref{npi4} for $(\pi_1,\pi_2,\pi_3, \overline{\sigma})$ are globally well-defined.
Indeed, noticing \eqref{npi1}, we can consider $(\pi_1,\pi_3, \overline{\sigma})$ firstly. Following  the proof  in Lemma \ref{lemma 10}, we obtain 

\begin{equation}\label{defpi} \boldsymbol{\pi}(x,t) = \left( \pi_1(x,t),\, \pi_3(x,t),\, \overline{\sigma}(x,t) \right)^\top, \end{equation} 
which satisfies \begin{equation}\label{eqpi} \partial_t \boldsymbol{\pi}(x,t) = \lambda \left( \int_{t-\tau}^{t} X(s)  ds \right) \partial_x \boldsymbol{\pi}(x,t), \end{equation} 
with the boundary condition \begin{equation}\label{bdpi} \boldsymbol{\pi}(0,t) = \left( X(t), \,  \int_{t-\tau}^{t} X(s)  ds,\, t \right)^\top. \end{equation} 
Using the method of characteristics, the solution to the system \eqref{defpi}--\eqref{bdpi} can be expressed as 
$\overline{\sigma}(x,t)=\xi^{-1}(x + \xi(t))$ and $\pi_1(x,t)=X(\overline{\sigma}(x,t))$, $\pi_3(x,t)=\int_{\overline{\sigma}(x,t) - \tau}^{\overline{\sigma}(x,t)} X(s) \, ds$.
		 The regularity of $\boldsymbol{\pi}$ is determined by the regularity of the functions that constitute its boundary data at $x=0$, namely $X$ and $\xi$. Given that $(X, \xi) \in (C^1([0,\infty)))^2$,  the components $\pi_1$, $\pi_3$, and $\overline{\sigma}$ are continuously differentiable. 
         
        After that, from \eqref{npi2}, since $w$ is the locally Lipschitz continuous solution to the target system, $\pi_2$ is obtained and also locally Lipschitz continuous.

        This set of regularities is sufficient for the well-posedness of the inverse backstepping transformation (\ref{non inverse backstepping}). Specifically, the expression for \textcolor{black}{$u$} involves $\pi_2$ only within the function $g(\cdot, \pi_2)$, which is then integrated. As $g$ is assumed to be Lipschitz continuous with respect to its second argument, the local Lipschitz continuity of $\pi_2$ ensures that the integral term is well-defined and that the resulting control input \textcolor{black}{$u$} is locally Lipschitz on $[0,D] \times [0,\infty)$.
		
		Finally, from \eqref{pp}--\eqref{ssi}, we obtain the well-posedness  of the forward predictor states $(p_1, p_2, p_3, \sigma)$. 
		
		Above all, we have proved the well-posedness of the closed-loop system (\ref{M1})--(\ref{M4}) with the controller (\ref{MF1})--(\ref{MF7}).


\section*{Acknowledgments}
The work of Peipei Shang and Li Tong is	supported by the National Natural Science Foundation of China (No. 12171368). The work of Mamadou Diagne was funded by the NSF CAREER Award CMMI-2302030 and  the NSF grant CMMI-2222250.

\bibliographystyle{abbrv}
\bibliography{ref}

\end{document}